\documentclass{amsart}

\usepackage{amsmath, amsthm, amssymb, amsfonts, mathtools}
\usepackage{mathrsfs, yfonts, dsfont}
\usepackage{enumitem}
\usepackage{tikz}
\usepackage{tikz-3dplot}
\usepackage{tikz-cd}
\usepackage{MnSymbol}
\usepackage{hyperref}
\hypersetup{
    colorlinks=false,
    pdftitle={On the construction of Quotient Spaces by Algebraic Foliations}
}
\usepackage{multirow}

\newcommand*{\sheafhom}{\mathscr{H}\kern -.5pt om}

\theoremstyle{definition}
\newtheorem{definition}{Definition}[subsection]

\newtheorem{example}[definition]{Example}

\theoremstyle{plain}
\newtheorem{theorem}[definition]{Theorem}
\newtheorem*{theoremA}{Theorem A}
\newtheorem*{theoremB}{Theorem B}
\newtheorem{corollary}[definition]{Corollary}
\newtheorem{proposition}[definition]{Proposition}
\newtheorem{lemma}[definition]{Lemma}

\theoremstyle{remark}
\newtheorem{remark}[definition]{Remark}

\numberwithin{equation}{definition}
 
\begin{document}

\title{On the construction of Quotient Spaces by Algebraic Foliations}
\author{Federico Bongiorno}
\address{Department of Mathematics, Imperial College London, 
180 Queen's Gate, London, SW7 2AZ, United Kingdom}
\email{federico.bongiorno14@imperial.ac.uk}

\begin{abstract}
Given a variety defined over a field of characteristic zero and an algebraically integrable foliation of corank less than or equal to two, we show the existence of a categorical quotient, defined on the non-empty open set of stable points, through which every invariant morphism factors uniquely.
\end{abstract}

\maketitle

\section{Introduction}
\label{sec:introduction}

An algebraically integrable foliation on a variety $X$ is a foliation induced by a rational map $X \dashedrightarrow Y$.
Different rational maps may induce the same foliation.
However, does an algebraically integrable foliation induce a canonical rational map?

\begin{theoremA}
Let $k$ be a field of characteristic zero and let $X$ be a normal integral scheme locally of finite type over $k$.
Suppose that $\mathscr{F}$ is an algebraically integrable foliation on $X$.
Let $X^{\mathrm{s}}$ denote the set of $\mathscr{F}$-stable points. Suppose that either
\begin{enumerate} 
\item \emph{(Theorem \ref{thm:main1}).} The first integrals of $\mathscr{F}$ are locally of finite type over $k$.
\item \emph{(Theorem \ref{thm:main2}).} The corank of $\mathscr{F}$ is less than or equal to two.
\end{enumerate}
Then $X^{\mathrm{s}}$ is a non-empty open set and there exists a geometric (and categorical) quotient of $X$ by $\mathscr{F}$, $\pi : X^{\mathrm{s}} \rightarrow Y$, where $Y$ is a normal integral scheme locally of finite type over $k$.
\end{theoremA}

The terminology used in the above theorem is motivated and explained in the next subsections by an extensive comparison with group actions and Geometric Invariant Theory (GIT).
For now, let us mention that a \emph{categorical quotient} of $X$ by $\mathscr{F}$ is a universal $\mathscr{F}$-invariant morphism (Definition \ref{def:categorical_quotient}), a \emph{geometric quotient} represents algebraic leaves, rather than orbits (Definition \ref{def:geom_quotients}) and a point is $\mathscr{F}$-stable if there exists a neighborhood whose induced morphism to its first integrals is smooth with connected $\mathrm{rank} \mathscr{F}$-dimensional fibres (Definition \ref{def:stability}). \par 

Throughout the course of the proof, there is one result which the author believes to be of independent interest.

\begin{theoremB}[Corollary \ref{cor:fin_gen}]
Let $k$ be a field of characteristic zero and let $X$ be a normal integral scheme locally of finite type over $k$.
Suppose that $\mathscr{F}$ is a torsion-free distribution and $\mathrm{corank}\,\mathscr{F} \leq 2$.
Then the first integrals of $\mathscr{F}$ are locally of finite type over $k$.
\end{theoremB}

\subsection{Motivation}

Studying morphisms between schemes is at the core of the classification problem in algebraic geometry.
To this end, consider the following association
$$
\left\{
\begin{array}{c}
\pi : X \rightarrow Y \, | \, \text{$\pi$ is morphism}\\
\text{of schemes over a scheme $S$}
\end{array}
\right\}
\xrightarrow{\quad\alpha\quad}
\left\{
\begin{array}{c}
\Omega_{X/S}^1 \rightarrow \mathscr{F} \, | \,  \text{$\mathscr{F}$ is a}\\
\text{quasi-coherent quotient sheaf}
\end{array}
\right\},
$$
where $\alpha$ takes a morphism $\pi$ to the surjective morphism of quasi-coherent sheaves
$$ \Omega_{X/S}^1 \rightarrow \Omega_{X/Y}^1. $$
The objects on the right-hand side are called \emph{distributions}.
$\alpha$ is not a bijective correspondence.
Nonetheless, one may ask the following two questions:
\begin{enumerate}
\item
What is the image of $\alpha$ and how can it be characterised?
\item
Does an inverse map exist?
\end{enumerate}
Question (1) is extremely challenging.
Up to torsion, these are called \emph{algebraically integrable foliations}.
Certainly, one necessary condition for a distribution $\mathscr{F}$ to be algebraically integrable is that the exterior derivative of the exterior algebra of $\Omega_{X/S}^1$ descends to the exterior algebra of $\mathscr{F}$ ($\mathscr{F}$ is a \emph{foliation}).
Ekedahl, Shepherd-Barron and Taylor conjecture that if $\mathscr{F}$ is also $p$-closed for all $p$ large enough, then, up to torsion, $\mathscr{F}$ is algebraically integrable.
Bost proves this in \cite{MR1863738} under additional hypotheses. \par

The answer to Question (2) is negative. For instance, the relative sheaf of differentials of any unramified morphism is zero.
Nonetheless, given a distribution $\mathscr{F}$, one may ask if, among all morphisms compatible with $\mathscr{F}$ there is a universal one.
More precisely, define a morphism $\pi : X \rightarrow Y$ to be \emph{$\mathscr{F}$-invariant} (Definition \ref{def:invariant_morphisms}) if there exists a factorisation
$$ \Omega_{X/S}^1 \rightarrow \Omega_{X/Y}^1 \rightarrow \mathscr{F}.$$

\textbf{Main question.}
Let $X$ be a scheme over $S$ and $\mathscr{F}$ a distribution on $X$ over $S$.
Under what assumption, does there exist an $\mathscr{F}$-invariant $S$-morphism $\pi : X \rightarrow Y$ such that for any $\mathscr{F}$-invariant $S$-morphism $\pi^{\prime} : X \rightarrow Y^{\prime}$, there exists a unique morphism $f : Y \rightarrow Y^{\prime}$ such that $\pi^{\prime} = f \circ \pi$? \par

If such a $\pi$ exists, it is unique and it is the \emph{categorical quotient} of $X$ by $\mathscr{F}$ (Definition \ref{def:categorical_quotient}).
As an example, if $\mathscr{F} = 0$, the categorical quotient is the identity morphism.

\begin{example}[Foliation by parabolas]
\label{ex:straight}
Let $k = \mathbb{C}$, $B = k[x,y]$ and $X = \mathbb{A}_k^2 = \mathrm{Spec}\,B$.
Define a foliation $\mathscr{F}$ to be the sheaf associated to the module $F$, where $F$ is the quotient of $\Omega_{X/k}^1 = B \langle dx \rangle \oplus B \langle dy \rangle$ by the submodule $(dy - 2x \cdot dx)$.
After dualising, this corresponds to the inclusion of the vector field
$$ \partial = \frac{\partial}{\partial x} + 2x \cdot \frac{\partial}{\partial y} $$
into the tangent sheaf of $X$. \par

The closed subvarieties $\iota_c : Z_c \hookrightarrow X$ defined by the equation $y = x^2 + c$, as $c$ varies in $k$, are everywhere tangent to $\partial$.
This can be seen from the fact that the pullback $\iota_c^*\,\Omega_{X/k}^1 \rightarrow \Omega_{Z_c/k}^1$ factors into
$$ \iota_c^*\,\Omega_{X/k}^1 \rightarrow \iota_c^*\,\mathscr{F} \rightarrow \Omega_{Z_c/k}^1.$$
When the induced map $\iota_c^*\,\mathscr{F} \rightarrow \Omega_{Z_c/k}^1$ is an isomorphism, $Z_c$ is an \emph{algebraic leaf} of $\mathscr{F}$. \par

The algebraic leaves $Z_c$ can be computed in another way.
Let $d_F$ denote the composition of the exterior derivative $d : B \rightarrow  \Omega_{B/k}^1$ with the quotient $\Omega_{B/k}^1 \rightarrow F$.
This is the \emph{foliated exterior derivative}.
Define
$$ A = \mathrm{ker} \left( B \xrightarrow{\quad d_F \quad} F \right).$$
$A$ is a subring of $B$ called the \emph{ring of first integrals} of $\mathscr{F}$.
In this case, it is given by $A = k[\,y - x^2\,]$.
Observe that the geometric fibres of the morphism
$$ \mathrm{Spec}\, B \rightarrow  \mathrm{Spec}\, A $$
are in bijective correspondence with the algebraic leaves $Z_c$, as $c$ varies in $k$.
\end{example}

A first attempt to construct the categorical quotient is to define a moduli functor of families of algebraic leaves (Definition \ref{def:functor_algebraic_leaves}) and construct the moduli space.
This is hard (Remark \ref{rem:moduli_space}).
Instead, one may try to classify only families of algebraic leaves over a point and hope it is enough to construct the categorical quotient.
This is the main idea of David Mumford's Geometric Invariant Theory (GIT) to construct the categorical quotient of a group acting on a scheme.
In fact, the two problems share many characteristics and the methods employed in this article are similar to either \cite{MR1304906} or \cite{MR1432041}.

\subsection{Group Actions and Foliations}

When taking a quotient of a variety by a group action, it is clear that the the field of fractions of the resulting variety should be equal to the field of rational $G$-invariant functions.
The aim of GIT is to choose a model among this birational equivalence class whose geometric points parametrise the set of orbits of the action.
Similarly, it is equally clear that when taking a quotient of a variety by a foliation, the field of fractions of the resulting variety should be equal to the field of \emph{rational first integrals} (or rational $\mathscr{F}$-invariant functions) (Definition \ref{def:local_first_integrals}).
Indeed, in \cite{MR1017286} and \cite[\S 2.3]{MR3842065}, the authors construct a quotient space, up to birational equivalence, which parametrises the general algebraic leaf.
However, in this birational equivalence class, there shall exist a quotient whose geometric points parametrise the algebraic leaves of the foliation.

To motivate the definition of a geometric quotient of $X$ by $\mathscr{F}$ (Definition \ref{def:geom_quotients}), let us make the analogy between group actions and foliations more precise and explore the relation between $G$-invariance and $\mathscr{F}$-invariance. \par

Let $X$ be a scheme over $S$. A group action $G$ on $X$ over $S$ induces an algebraically integrable foliation on $X$ over $S$ in the following way.
Let $p_1$ and $p_2$ denote the projections of $X \times_S X$ to the first and second factor respectively.
A group action is a morphism $G \times_S X \rightarrow X$.
This induces a morphism $G \times_S X \rightarrow X \times_S X$ whose composition with $p_1$ is the group action and whose composition with $p_2$ is the projection of $G \times_S X$ to the second factor.
Let $X \rightarrow G \times_S X$ be the morphism whose components are the constant morphism to the identity of the group $e_G$ and the identity morphism to $X$.
Using the properties of the group action, it follows that the diagonal morphism $\Delta : X \rightarrow X \times_S X$ factors into
$$ X \rightarrow G \times_S X \rightarrow X \times_S X.$$
Let $R_G$ denote the scheme-theoretic image of $G \times_S X \rightarrow X \times_S X$.
There is a sequence of morphisms
$$ X \rightarrow R_G \rightarrow X \times_S X.$$
The conormal sheaf of $\Delta$ is $\Omega_{X/S}^1$.
By \cite[\href{https://stacks.math.columbia.edu/tag/07RK}{Lemma 07RK}]{stacks-project}, $X \rightarrow R_G$ is an immersion.
Let $\mathscr{F}_G$ denote its conormal sheaf.
By \cite[\href{https://stacks.math.columbia.edu/tag/062S}{Lemma 062S}]{stacks-project} , there is a surjection
$$ \Omega_{X/S}^1 \rightarrow \mathscr{F}_G.$$
Hence $\mathscr{F}_G$ is a distribution on $X$ over $S$.
It is easy to see that it is in fact an algebraically integrable foliation. \par

A local section of $\mathscr{O}_X$ is $G$-invariant if it is in the kernel of
\begin{align}
\label{eq:group_kernel}
\mathscr{O}_X \xrightarrow{p_1^* - p_2^*} {p_2}_* \mathscr{O}_{G \times_S X}.
\end{align}
For a distribution $q : \Omega_{X/S}^1 \rightarrow \mathscr{F}$, a local section of $\mathscr{O}_X$ is $\mathscr{F}$-invariant (or it is a first integral) if it is in the kernel of
\begin{align}
\label{eq:foliation_kernel}
\mathscr{O}_X \xrightarrow{q \circ d} \mathscr{F}.
\end{align}
If $S$ is the spectrum of a field, $G$ is geometrically reduced over $S$ and $X$ is reduced and separated, then a $G$-invariant local section of $\mathscr{O}_X$ is also $\mathscr{F}_G$-invariant.
The converse is not true: any action by a finite group induces a torsion foliation. Therefore, any section of $\mathscr{O}_X$ defined over an open set not intersecting the support of $\mathscr{F}_G$ is $\mathscr{F}_G$-invariant.

\subsection{Comparison of Geometric Quotients}
Now let $\mathscr{F}$ be an arbitrary algebraically integrable foliation on $X$ over $S$ and let $\pi : X \rightarrow Y$ be an $S$-morphism.
The following table compares the definitions of geometric quotient of a scheme $X$ by a group action and by a foliation.
\begin{center}
\bgroup
\def\arraystretch{1.4}
\begin{tabular}{ |p{59mm}|p{59mm}|}
\hline
\multicolumn{2}{|c|}{\textbf{Geometric Quotient} }\\
\hline
\textbf{By a group $G$ acting on $X$}& \textbf{By a foliation $\mathscr{F}$ on $X$}\\
\hline
\hline
$\pi$ is $G$-invariant& $\pi$ is $\mathscr{F}$-invariant\\
\hline
Geometric points of $\pi$ represent set of orbits& Geometric
points of $\pi$ represent set of algebraic leaves \\
\hline
$\pi$ is universally open& $\pi$ is universally open\\
\hline
$0 \rightarrow \mathscr{O}_Y \rightarrow \pi_* \mathscr{O}_X \rightarrow \pi_* \mathscr{O}_{G \times_S X}$ is exact& $0 \rightarrow \pi^{-1} \mathscr{O}_Y \rightarrow \mathscr{O}_X \rightarrow \mathscr{F}$ is exact  \\
\hline
\end{tabular}
\egroup
\end{center}

Here the morphism $\pi_* \mathscr{O}_X \rightarrow \pi_* \mathscr{O}_{G \times_S X}$ is the pushforward of the morphism in (\ref{eq:group_kernel}) and the morphism $\mathscr{O}_X \rightarrow \mathscr{F}$ is the morphism in (\ref{eq:foliation_kernel}).
This is where the greatest difference between the two definitions lies.
The former is local on $Y$, the latter is local on $X$.
The former does not behave well under localisation on $X$ (Example \ref{ex:radii}), whereas the latter allows for glueing (Lemma \ref{lem:open_geometric_quotient}).
On the other hand, the latter definition behaves poorly under arbitrary base change. \par

As with GIT, a geometric quotient of $X$ by $\mathscr{F}$ is the categorical quotient of $X$ by $\mathscr{F}$ (Proposition \ref{prop:geometric_categorical_quotient}). In particular, it is unique.
The strategy employed to construct the geometric quotient of $X$ by $\mathscr{F}$ is basic: compute the ring of first integrals over suitable open sets and glue.
There are two main difficulties. It has to be shown that, over suitable open sets, the rings of first integrals of $\mathscr{F}$ 
\begin{enumerate}
\item Are finitely generated.
\item Can be glued together.
\end{enumerate}
As with GIT, the rings of first integrals of a foliation need not be finitely generated (Remark \ref{rem:fin_gen}).
On the other hand, we were able to show finite generation when the corank of $\mathscr{F}$ is less that or equal to two and $X$ is normal (Corollary \ref{cor:fin_gen}).
In general, a guiding principle is that a reductive group $G$ should give rise to a foliation with at most log-canonical singularities and foliations with at most log-canonical singularities should have finitely generated first integrals. \par

With regards to (2), the notion of \emph{stability} is introduced (Definition \ref{def:stability}).
This allows for glueing first integrals (Proposition \ref{prop:stable_geometric quotient}).
The nomenclature reflects the corresponding notion in GIT, however it does not depend on a linearisation, instead it depends on the properties of the first integrals in a neighborhood of a point.
It does not seem easy to check wether a given point is $\mathscr{F}$-stable, nonetheless we were able to show that, in good circumstances, the generic point of an integral scheme is $\mathscr{F}$-stable (Proposition \ref{prop:generic_stability}).

\begin{example}[Foliation by radii]
\label{ex:radii}
Let $k$ be a field of characteristic zero, let $B = k[x,y]$ and let $X = \mathrm{Spec} \, B$.
Suppose that $k^{\times}$ acts on $X$ with weights $(1,1)$.
This induces a foliation $\mathscr{F}$ on $X$.
Note that $\mathscr{F}$ is the sheaf associated to the module $F$, where $F$ is the quotient of $\Omega_{X/k}^1 = B \langle dx \rangle \oplus B \langle dy \rangle$ by the submodule $(x \cdot dy - y \cdot dx)$.
After dualising, this corresponds to the inclusion of the vector field
vector field $x \partial/\partial x + y \partial/\partial y$ into the tangent sheaf of $X$.
This is represented in the figure below.
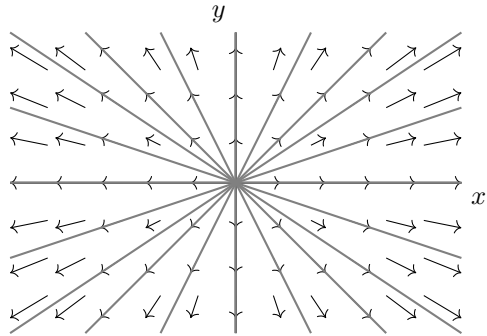
\begin{center}
\begin{tikzpicture}
\draw[thick] (-3,0) -- (3,0) node[anchor=north west] {$x$};
\draw[thick] (0,-2) -- (0,2) node[anchor=south east] {$y$};
\foreach \x in {-2.5,-2,-1.5,-1,-0.5,0.5,1,1.5,2,2.5}
  \foreach \y in {-1.5,-1,-0.5,0,0.5,1,1.5}
    \draw[->] (\x,\y) -- (1.2*\x,1.2*\y);
\foreach \y in {-1.5,-1,-0.5,0.5,1,1.5}
  \draw[->] (0,\y) -- (0,1.2*\y);
\draw[thick, gray] (-3,0) -- (3,0);
\draw[thick, gray] (-3,1) -- (3,-1);
\draw[thick, gray] (-3,2) -- (3,-2);
\draw[thick, gray] (-2,2) -- (2,-2);
\draw[thick, gray] (-1,2) -- (1,-2);
\draw[thick, gray] (0,2) -- (0,-2);
\draw[thick, gray] (1,2) -- (-1,-2);
\draw[thick, gray] (2,2) -- (-2,-2);
\draw[thick, gray] (3,2) -- (-3,-2);
\draw[thick, gray] (3,1) -- (-3,-1);
\end{tikzpicture}
\end{center}
Let us point out the difference between orbits and algebraic leaves.
The orbits of the action are given by the origin $(0,0)$ and the lines $c_x x + c_y y = 0$ without the origin, where $c_x$ and $c_y$ are not both zero in $k$.
On the other hand, there are no algebraic leaves.
It may appear that the closed subscheme $Z$ cut out by, say, $y = 0$ is an algebraic leaf.
However, the morphism
$$ F|_Z = \frac{k[x] \langle dx \rangle \oplus k[x] \langle dy  \rangle }{(x \cdot dy)} \rightarrow k[x] \langle dx \rangle = \Omega_{Z/k}^1 $$
is not an isomorphism.
The problem is that the foliation has a singularity at the origin. Declare the origin to be \emph{unstable} and remove it.
Let $X^{\mathrm{s}}$ denote the complement of the origin in $X$.
The closed subschemes of $X^{\mathrm{s}}$ cut out by $c_x x + c_y y = 0$, where $c_x$ and $c_y$ are not both zero in $k$, are now algebraic leaves, as well as orbits for the action of $k^{\times}$. \par

Observe that the first integrals of $F$ (or the $k^{\times}$-invariant functions) over $X$ are only the constant functions in $k$.
The same is true when restricting to $X^{\mathrm{s}}$.
However, when restricting to the distinguished affine open set $D(x)$ where $x$ is non-zero, one can compute that
$$ d \left( \frac{y}{x} \right) = \frac{1}{y^2} (y \cdot dx - x \cdot dy) = 0 \in F.$$
Hence $\frac{y}{x}$ is a first integral of $F$ over $D(x)$.
More precisely, $k\left[\frac{y}{x}\right]$ is the ring of first integrals of $F$ over $D(x)$.
Observe that the ring of first integrals over $D(x)$ is not a localisation of the ring of first integrals over $X$.
Similarly, $k\left[\frac{x}{y}\right]$ is the ring of first integrals of $F$ over $D(y)$.
On the intersection $D(xy)$, the ring of first integrals is $k\left[\frac{y}{x},\frac{x}{y}\right]$.
Since $D(x)$ and $D(y)$ cover $X^{\mathrm{s}}$, the first integrals can be glued together along their intersection to obtain a quotient morphism
$$ \pi : \mathbb{A}_k^2 \setminus \{0\} \rightarrow \mathbb{P}_k^1.$$
The fibres of $\pi$ are in bijection with the set of algebraic leaves and one can show that $\pi$ is the geometric and categorical quotient of $X^{\mathrm{s}}$ by $\mathscr{F}$.
\end{example}

\begin{example}[Foliation by hyperbolae]
\label{ex:hyperbolae}
Let $k$ be a field of characteristic zero, let $B = k[x,y]$ and $X = \mathrm{Spec} \, B$.
Suppose that $k^{\times}$ acts on $X$ with weights $(1,-1)$.
This induces a foliation $\mathscr{F}$ on $X$.
Note that $\mathscr{F}$ is the sheaf associated to the module $F$, where $F$ is the quotient of $\Omega_{X/k}^1 = B \langle dx \rangle \oplus B \langle dy \rangle$ by the submodule $(x \cdot dy + y \cdot dx)$.
After dualising, this corresponds to the inclusion of the vector field
vector field $x \partial/\partial x - y \partial/\partial y$ into the tangent sheaf of $X$.
This is represented in the figure below.
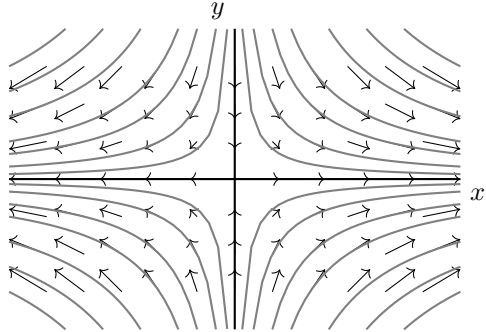
\begin{center}
\begin{tikzpicture}
\draw[thick] (-3,0) -- (3,0) node[anchor=north west] {$x$};
\draw[thick] (0,-2) -- (0,2) node[anchor=south east] {$y$};
\foreach \x in {-2.5,-2,-1.5,-1,-0.5,0.5,1,1.5,2,2.5}
  \foreach \y in {-1.5,-1,-0.5,0,0.5,1,1.5}
    \draw[->] (\x,\y) -- (1.2*\x,0.8*\y);
\foreach \y in {-1.5,-1,-0.5,0.5,1,1.5}
  \draw[->] (0,\y) -- (0,0.8*\y);
\foreach \k in {-4.5,-3.5,-2.5,-1.5,-1,-0.5,-0.2}
  \draw[thick, gray] (0,0) plot[domain=-3:(\k/2)] (\x,\k/\x);
\foreach \k in {-4.5,-3.5,-2.5,-1.5,-1,-0.5,-0.2}
  \draw[thick, gray] (0,0) plot[domain=-3:(\k/2)] (\x,-\k/\x);
\foreach \k in {4.5,3.5,2.5,1.5,1,0.5,0.2}
  \draw[thick, gray] (0,0) plot[domain=(\k/2):3] (\x,\k/\x);
\foreach \k in {4.5,3.5,2.5,1.5,1,0.5,0.2}
  \draw[thick, gray] (0,0) plot[domain=(\k/2):3] (\x,-\k/\x);
\end{tikzpicture}
\end{center}
This example shows a different phenomenon.
The orbits of the action are the hyperbolae given by the equation $xy = c$ for some non-zero $c \in k$, the punctured $x$-axis, the punctured $y$-axis and the origin.
On the other hand, only the non-trivial hyperbolae are algebraic leaves.
Indeed, the closed subscheme $Z$ cut out by $xy=0$ has the property that $\Omega_{Z/k}^1 = F|_Z$, however, it is not integral. \par

The ring of first integrals over $X$ gives an apparent satisfactory answer: the inclusion of rings $k[xy] \rightarrow k[x,y].$
Note that the non-special fibres of the morphism are non-degenerate hyperbolae, whereas the special fibre is the union of the $x$-axis and the $y$-axis.
It thus identifies three orbits of the action.
The induced morphism of spectra $X \rightarrow \mathbb{A}_k^1$ is a categorical quotient of $X$ by $\mathscr{F}$ but it is not a geometric quotient of $X$ by $\mathscr{F}$ since its fibres are not integral.
Requiring algebraic leaves to be integral ensures that the geometric quotient of an open subset $U \subseteq X$ by $\mathscr{F}$ is an open subset of the geometric quotient of $X$ by $\mathscr{F}$ (Lemma \ref{lem:open_geometric_quotient}).
Once again, the issue is the origin.
The singularity of the foliation does not allow one to separate the two algebraic leaves meeting at the origin.
Declare the origin to be unstable and remove it.
Let $X^{\mathrm{s}}$ denote the complement of the origin in $X$.
Now, there are two more algebraic leaves given by the punctured $x$-axis and the punctured $y$-axis.
One can compute the ring of first integrals of $F$ over $D(x)$, $D(y)$ and $D(xy)$ to be respectively $k[xy]$, $k[xy]$ and $k\left[ xy, \frac{1}{xy} \right]$.
Glueing the rings gives rise to a quotient morphism
$$ \pi : \mathbb{A}_k^2 \setminus \{0\} \rightarrow Y,$$
where $Y$ is the affine line with a double origin.
The fibres of $\pi$ over the two origins correspond to the $x$-axis and the $y$-axis.
Whilst $X^{\mathrm{s}}$ is an open subset of $X$, $Y$ is not an open subset of $\mathbb{A}_k^1$.
\end{example}

\begin{example}[Foliation by hyperbolae in three dimensions]
\label{ex:hyperbolae3d}
Let $k$ be a field of characteristic zero, let $B = k[x,y,z]$ and let $X = \mathrm{Spec} \, B$.
Suppose that $k^{\times}$ acts on $X$ with weights $(-1,1,1)$.
This induces a foliation on $X$ whose corresponding vector field is $-x \partial/\partial x + y \partial/\partial y + z \partial/\partial z$.
This example combines both previous examples.
The leaf (or orbit) passing through a point not lying in the closed subscheme defined by the ideal $(xy,xz)$ is an hyperbola.
On the other hand, when either $x = 0$ or both $y = 0$ and $z = 0$, the foliation is represented in the figure below.
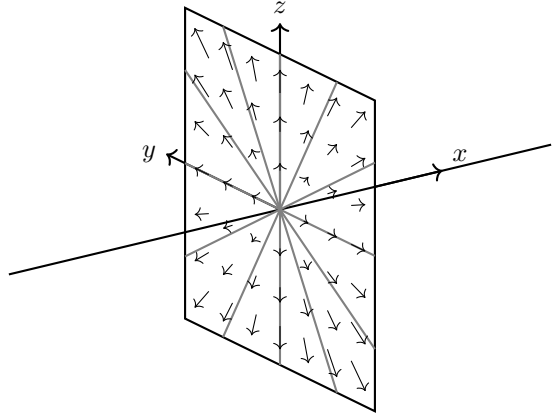
\begin{center}
\tdplotsetmaincoords{70}{145}
\begin{tikzpicture} [scale=2.2, tdplot_main_coords, axis/.style={->,thick}]

\draw[axis] (0,0,0) -- (-1.2,0,0) node[anchor=south west]{$x$};
\draw[axis] (0,0,0) -- (0,-1.2,0) node[anchor=east]{$y$};
\draw[axis] (0,0,0) -- (0,0,1.2) node[anchor=south]{$z$};

\draw[thick,tdplot_main_coords] (0,1,1)-- (0,-1,1) -- (0,-1,-1)--(0,1,-1)--cycle;
\draw[thick,tdplot_main_coords] (0,0,0) -- (2,0,0);
\draw[thick,tdplot_main_coords] (-0.7,0,0) -- (-2,0,0);

\draw[thick, gray, tdplot_main_coords] (0,-0.6,-1) -- (0,0.6,1);
\draw[thick, gray, tdplot_main_coords] (0,-1,-0.6) -- (0,1,0.6);
\draw[thick, gray, tdplot_main_coords] (0,0.6,-1) -- (0,-0.6,1);
\draw[thick, gray, tdplot_main_coords] (0,1,-0.6) -- (0,-1,0.6);
\draw[thick, gray, tdplot_main_coords] (0,0,-1) -- (0,0,1);
\draw[thick, gray, tdplot_main_coords] (0,-1,0) -- (0,1,0);

\foreach \y in {-0.75,-0.5,-0.25,0.25,0.5,0.75}
  \foreach \z in {-0.75,-0.5,-0.25,0,0.25,0.5,0.75}
    \draw[->, tdplot_main_coords] (0,\y,\z) -- (0,1.2*\y,1.2*\z);
\foreach \z in {-0.75,-0.5,-0.25,0.25,0.5,0.75}
  \draw[->, tdplot_main_coords] (0,0,\z) -- (0,0,1.2*\z);
\end{tikzpicture}
\end{center}
Declaring the origin to be unstable allows to decompose the closed subscheme cut out by $xy = 0$ and $xz = 0$ into algebraic leaves.
Glueing the rings of first integrals over the distinguished open sets $D(x)$, $D(y)$ and $D(z)$ gives rise to a quotient morphism
$$ \pi : \mathbb{A}_k^3 \setminus \{0\} \rightarrow Y,$$
where $Y$ is the affine plane with a double origin of which one blown up.
The fibre over the origin which was not blown up is the algebraic leaf given by the $x$-axis.
The fibre over a point in the exceptional divisor is the algebraic leaf given by a punctured line in the $x = 0$ plane passing through the origin.
\end{example}

\subsection{Outline}
The article is organised as follows.
In \S \ref{sec:foliations}, the basic objects of study are introduced. These are distributions, foliations, tangent subvarieties, algebraic leaves and invariant morphisms.
In some cases, the definitions are non-standard, however, when this happens, it is pointed out in a subsequent remark.
In \S \ref{sec:quotient_spaces}, the notions of categorical and geometric quotients are defined.
It is shown that a geometric quotient is categorical and that geometric quotients can be glued together.
In \S \ref{sec:first_integrals}, the notion of first integral is introduced and their finite generation is proven in the special case above.
In \S \ref{sec:stability}, a stability condition is introduced and it is shown that there exists a geometric quotient in a neighborhood of a stable point. Furthermore, it is shown that the generic point is stable. 
In \S \ref{sec:main_theorems}, the main theorems are stated and proven.

\subsection{Acknowledgements}
I would firstly like to thank my advisor Paolo Cascini for the many insightful discussions held throughout these years and for introducing me to the Minimal Model Program and algebraic foliations.
I thank Jorge Vitorio Pereira for pointing out the counterexamples to finite generation and for suggesting the analogy with GIT.
I also thank Richard Thomas for introducing me to GIT.
I thank K\c{e}stutis \u{C}esnavi\u{c}ius for help in the argument of the last paragraph of Proposition \ref{prop:generic_stability}.
I thank Paolo Cascini, Jorge Vitorio Pereira and Liam Stigant for reading a draft and suggesting several improvements.
Finally, I thank Federico Barbacovi, Fabio Bernasconi, Se\'{a}n Keel, Wendalin Lutz, Mirko Mauri, Michael McQuillan, Jaime Mendizabal Roche, Nicholas Shepherd-Barron, Calum Spicer and Angelo Vistoli for taking the time to answer my questions.
I acknowledge funding from the EPSRC for my PhD bursary.

\section{Foliated Spaces}
\label{sec:foliations}

In this section, the basic objects of study are introduced.
\S \ref{subsec:foliations} defines distributions and foliations.
\S \ref{subsec:foliatied_spaces} defines directed and foliated spaces, a compact way to work with schemes endowed with a distribution or a foliation respectively.
\S \ref{subsec:invariant} introduces invariant morphisms.
\S \ref{subsec:algebraic_leaves} defines families of tangent subvarieties, algebraic leaves and their moduli problem.

\subsection{Foliations}
\label{subsec:foliations}

This subsection defines distributions, involutivity and foliations. It also defines the pullback of a distribution and shows it is well-defined.

\begin{definition}[Distributions]
\label{def:distributions}
Let $S$ be a scheme and let $X$ be a scheme over $S$.
A \emph{distribution} $\mathscr{F}$ on $X$ over $S$ is a quasi-coherent quotient sheaf of $\Omega_{X/S}^1$.
\end{definition}

\begin{remark}[Subsheaf distributions]
\label{rem:distribution_corr}
When $X$ is integral, a distribution is typically defined as a saturated subsheaf $\mathscr{T}_{\mathscr{F}}$ of the tangent sheaf $\mathscr{T}_{X/S}$. When $X$ is also locally Noetherian, there is a bijective correspondence
$$
\left\{
\begin{array}{c}
\mathscr{F} \;|\; \Omega_{X/S}^1 \rightarrow \mathscr{F} \rightarrow 0 \\
\text{and $\mathscr{F}$ is torsion-free}
\end{array}
\right\}
\begin{array}{c}
\xrightarrow{\quad\alpha\quad} \\
\xleftarrow[\quad\beta\quad]{}
\end{array}
\left\{
\begin{array}{c}
\mathscr{T}_{\mathscr{F}} \;|\; 0 \rightarrow \mathscr{T}_{\mathscr{F}} \rightarrow \mathscr{T}_{X/S} \\
\text{and $\mathscr{T}_{\mathscr{F}}$ is saturated}
\end{array}
\right\},
$$
where $\alpha$ and $\beta$ are defined as follows: for a quotient distribution $\mathscr{F}$,
$$\alpha(\mathscr{F}) = \sheafhom \left( \mathscr{F}, \mathscr{O}_X \right).$$
For a subsheaf distribution $\mathscr{T}_{\mathscr{F}}$, 
$$\beta(\mathscr{T}_{\mathscr{F}}) = \mathrm{im}\left(\Omega_{X/S}^1 \rightarrow \Omega_{X/S}^{[1]} \rightarrow \sheafhom \left( \mathscr{T}_{\mathscr{F}}, \mathscr{O}_X \right) \right).$$
To see that $\alpha$ and $\beta$ are inverse to each other, it is important to use the fact that, on an integral locally Noetherian scheme, the tangent sheaf, being the dual of a coherent sheaf, is reflexive
(\cite[\href{https://stacks.math.columbia.edu/tag/0AY4}{Lemma 0AY4}]{stacks-project}).
In Definition \ref{def:distributions}, a distribution is not required to be torsion-free since torsion cannot be defined on non-integral schemes and does not behave well under pullback.
\end{remark}

\begin{remark}[Coherence of distributions]
\label{rem:coherence_distributions}
Suppose that $X$ is locally of finite type over a locally Noetherian scheme $S$. Then $X$ is locally Noetherian (\cite[\href{https://stacks.math.columbia.edu/tag/01T6}{Lemma 01T6}]{stacks-project}) and the cotangent sheaf $\Omega_{X/S}^1$ is coherent (\cite[\href{https://stacks.math.columbia.edu/tag/01V2}{Lemma 01V2}]{stacks-project} and \cite[\href{https://stacks.math.columbia.edu/tag/01XZ}{Lemma 01XZ}]{stacks-project}).
This implies that any quasi-coherent quotient sheaf $\mathscr{F}$ of $\Omega_{X/S}^1$ is coherent (\cite[\href{https://stacks.math.columbia.edu/tag/01Y1}{Lemma 01Y1}]{stacks-project} and \cite[\href{https://stacks.math.columbia.edu/tag/01XZ}{Lemma 01XZ}]{stacks-project}). 
\end{remark}

\begin{definition}[Foliations]
\label{def:foliations}
Let $S$ be a scheme and let $X$ be a scheme over $S$. A distribution $\mathscr{F}$ on $X$ over $S$ is \emph{involutive} if there exists an open dense subset $U \subseteq X$ such that the total exterior derivative $d$, seen as an endomorphism of 
$$ \Omega_{X/S}^{\bullet}  = \bigoplus_{n > 0} \Omega_{X/S}^n, $$
descends to an endomorphism $ d_{\mathscr{F}}$ of
$$ \Lambda^{\bullet} \mathscr{F}|_U = \bigoplus_{n > 0} \Lambda^n \, \mathscr{F}|_U$$
over $U \subseteq X$.
A \emph{foliation} is an involutive distribution and $d_{\mathscr{F}}$ is the \emph{total exterior derivative} of the foliation.
\end{definition}

\begin{remark}[Subsheaf foliations]
\label{rem:foliation_corr}
When $X$ is integral and locally Noetherian, the one-form involutivity criterion can be used to show that, under the correspondence of Remark \ref{rem:distribution_corr}, a torsion-free quotient distribution is involutive if and only if the corresponding saturated subsheaf distribution is closed under the Lie bracket.
\end{remark}

\begin{remark}[Uniqueness of $d_{\mathscr{F}}$]
\label{rem:uniqueness_exterior_derivative}
Observe that, for all $n \in \mathbb{N}$, the morphism
$$ \Omega_{X/S}^n \rightarrow \Lambda^n \, \mathscr{F} $$
is surjective (\cite[\href{https://stacks.math.columbia.edu/tag/01CJ}{Lemma 01CJ}]{stacks-project}).
This implies that there is a surjection
$$ \Omega_{X/S}^{\bullet}  \rightarrow \Lambda^{\bullet} \mathscr{F}.$$
Hence, if $d_{\mathscr{F}}$ exists, it is unique.
In particular, since $d$ can be decomposed into morphisms
$$ d_n : \Omega_{X/S}^n \rightarrow \Omega_{X/S}^{n+1} $$
for $n \in \mathbb{N}$, whenever $d_{\mathscr{F}}$ exists, it must also decompose into morphisms
$$d_{\mathscr{F},n} : \Lambda^n \, \mathscr{F} \rightarrow \Lambda^{n+1} \, \mathscr{F}.$$
\end{remark}

\begin{definition}[Pullback of distributions]
\label{def:pullback}
Let $S$ be a scheme and let $X$ and $Y$ be schemes over $S$.
Let $f : X \rightarrow Y$ be an $S$-morphism and let $\mathscr{G}$ be a distribution on $Y$ over $S$.
Then the \emph{pullback distribution} $f^{\#} \mathscr{G}$ is the pushout of the diagram
$$ f^* \mathscr{G} \leftarrow f^* \Omega_{Y/S}^1 \rightarrow \Omega_{X/S}^1.$$
\end{definition}

\begin{lemma}[Pullback is well-defined]
\label{lem:pullback}
Let $S$ be a scheme and let $X$ and $Y$ be schemes over $S$.
Let $f : X \rightarrow Y$ be an $S$-morphism and let $\mathscr{G}$ be a distribution on $Y$ over $S$.
Then $f^{\#} \mathscr{G}$ is a distribution.
Furthermore, suppose that $\mathscr{G}$ is a foliation on $Y$ over $S$ and at least one of the following holds.
\begin{itemize}
\item $X$ is irreducible and $f$ is dominant. 
\item $f$ is an open morphism.
\end{itemize}
Then $f^{\#} \mathscr{G}$ is a foliation on $X$ over $S$.
\end{lemma}

\begin{proof}
Since $f^*$ is right exact, $f^* \Omega_{Y/S}^1 \rightarrow f^* \mathscr{G}$ is surjective. This implies that the pushout $\Omega_{X/S}^1 \rightarrow f^{\#} \mathscr{G}$ is surjective.
Furthermore, because the category of quasi-coherent sheaves is closed under colimits, $f^{\#} \mathscr{G}$ is quasi-coherent, hence it is a distribution. \par

It remains to show that, under the assumptions above, $f^{\#} \mathscr{G}$ is involutive whenever $\mathscr{G}$ is involutive.
Let $V \subseteq Y$ be a dense open subset of $Y$ such that the restriction of $\mathscr{G}$ to $V$ admits a total exterior derivative.
If either $X$ is irreducible and $f$ is dominant or $f$ is open, the inverse image of $V$ is a dense open subset of $X$. Up to replacing $Y$ by $V$, it may be assumed that $\mathscr{G}$ admits a total exterior derivative.
Firstly, note that $d_{f^{\#} \mathscr{G},0}$ always exists and it is given as the composition of $d_0$ with the quotient morphism $\Omega_{X/S}^1 \rightarrow f^{\#} \mathscr{G}$.
Next, it is shown that
$$ d_{f^{\#} \mathscr{G},1} : f^{\#} \mathscr{G} \rightarrow \Lambda^2 f^{\#} \mathscr{G} $$
exists. To this end, consider the diagram 
\begin{center}
\begin{tikzcd}[row sep=scriptsize, column sep=scriptsize]
& f^* \Omega_{Y/S}^2 \arrow[rr] \arrow[dd] & & \Omega_{X/S}^2 \arrow[dd] \\
f^* \Omega_{Y/S}^1 \arrow[ur, "d_1"] \arrow[rr, crossing over] \arrow[dd] & & \Omega_{X/S}^1 \arrow[ur, "d_1"] \\
& f^*\Lambda^2\, \mathscr{G} \arrow[rr] & & \Lambda^2\, f^{\#} \mathscr{G} \\
f^*\mathscr{G} \arrow[ur, "d_{\mathscr{G},1}"] \arrow[rr] & & f^{\#} \mathscr{G} \arrow[from=uu, crossing over],
\end{tikzcd}
\end{center}

where the commuting square in the background is obtained by applying the second exterior power to the commuting square in the foreground, and noting that exterior powers commute with pullbacks.
Recall that $f^{\#} \mathscr{G}$ is the pushout of the diagram
$$ f^* \mathscr{G} \leftarrow f^* \Omega_{Y/S}^1 \rightarrow \Omega_{X/S}^1$$
in the category of quasi-coherent $\mathscr{O}_X$-modules.
The inclusion functor from the category of quasi-coherent $\mathscr{O}_X$-modules to the category of abelian sheaves over $X$ is exact (\cite[\href{https://stacks.math.columbia.edu/tag/01BY}{Lemma 01BY}]{stacks-project}).
Hence $f^{\#} \mathscr{G}$ is the pushout in the category of abelian sheaves as well.
Existence of $d_{f^{\#}\mathscr{G},1}$ then follows from the universal property of the pushout.
Finally $d_{f^{\#}\mathscr{G}}$ is constructed.
By Remark \ref{rem:uniqueness_exterior_derivative}, if $d_{f^{\#}\mathscr{G}}$ exists, it is unique.
This implies that it is enough to construct it locally.
Let $X = \mathrm{Spec}\,B$ and let $F = \Gamma \left( X, \mathscr{F} \right)$.
Constructing an exterior derivative $d_F$ on the exterior algebra $\Lambda^{\bullet} F$ is equivalent to constructing compatible $d_{F,0} : B \rightarrow F$ and $d_{F,1} : F \rightarrow \Lambda^2 F$ (\cite[Chapter III, \S 10.9, Proposition 14]{MR0274237}).
Thus involutivity of $f^{\#} \mathscr{G}$ is proven.
\end{proof}

\begin{lemma}[Composition of pullbacks]
\label{lem:pullback_functorial}
Let $S$ be a scheme and let $X$, $Y$ and $Z$ be schemes over $S$.  Let $f : X \rightarrow Y$ and $g : Y \rightarrow Z$ be $S$-morphisms and let $\mathscr{H}$ be a distribution on $Z$ over $S$.
Then
$$ (g f)^{\#} \mathscr{H} = f^{\#} ( g^{\#} \mathscr{H} ) .$$
\end{lemma}

\begin{proof}
This is an easy consequence of the fact that $f^*$ preserves colimits and colimits commute with colimits.
\end{proof}

\begin{lemma}[Pullback by weakly \'{e}tale morphisms]
\label{lem:etale_pullback}
Let $S$ be a scheme and let $X$ and $Y$ be schemes over $S$.
Let $f : X \rightarrow Y$ be a weakly \'{e}tale $S$-morphism and let $\mathscr{G}$ be a distribution on $Y$ over $S$.
Then
$$ f^{\#} \mathscr{G} = f^* \mathscr{G}.$$
Furthermore, if $\mathscr{G}$ is a foliation and $f$ is open, $f^{\#} \mathscr{G}$ is a foliation.
\end{lemma}

\begin{proof}
$f^{\#} \mathscr{G}$ is the  pushout of the diagram
$$ f^* \mathscr{G} \leftarrow f^* \Omega_{Y/S}^1 \rightarrow \Omega_{X/S}^1.$$
Since $f$ is weakly \'{e}tale, $f^* \Omega_{Y/S}^1 \rightarrow \Omega_{X/S}^1$ is an isomorphism (\cite[\href{https://stacks.math.columbia.edu/tag/08R2}{Lemma 08R2}]{stacks-project}).
Therefore
$$ f^{\#} \mathscr{G} = f^* \mathscr{G}.$$
If $\mathscr{G}$ is a foliation and $f$ is open, by Lemma \ref{lem:pullback}, $f^{\#} \mathscr{G}$ is a foliation.
\end{proof}

\subsection{Foliated Spaces}
\label{subsec:foliatied_spaces}

This subsection defines a suitable category to study distributions. 
To this end, it is necessary to first discuss how to rule out some pathologies occurring in positive characteristic.

\begin{definition}[Residually separable morphisms]
\label{def:separable_morphism}
Let $f : X \rightarrow Y$ be a morphism of schemes.
Then $f$ is \emph{residually separable} if for all $x \in X$ the induced inclusion of residue fields
$$ \kappa \left(f(x) \right) \subseteq \kappa(x) $$
is separable.
It is \emph{universally separable} if for all morphisms $Y^{\prime} \rightarrow Y$, the base change $f^{\prime} : X \times_Y Y^{\prime} \rightarrow Y^{\prime}$ is residually separable.
\end{definition}

\begin{remark}[Characterisation of universal separability]
\label{rem:char_separability}
Suppose that $f : X \rightarrow Y$ is a morphism of schemes.
Using \cite[\href{https://stacks.math.columbia.edu/tag/01JT}{Lemma 01JT}]{stacks-project}, it follows that $f$ is universally separable if and only if, for all $x \in X$ and for all field extensions $\kappa(f(x)) \subseteq L$, the inclusion morphism $ L \rightarrow \kappa(x) \otimes_{\kappa(f(x))} L$ is residually separable.
\end{remark}

\begin{remark}[Separability in characteristic zero]
\label{rem:char_zero_separable}
If $f : X \rightarrow Y$ is a morphism of schemes defined over a over a field of characteristic zero, then it is universally separable.
\end{remark}

\begin{definition}[Foliated spaces]
\label{def:foliated_space}
Let $S$ be a scheme.
A \emph{directed space} over $S$ is a pair $(X, \mathscr{F})$ where $X$ is a scheme over $S$ and $\mathscr{F}$ is a distribution on $X$ over $S$. 
If $\mathscr{F}$ is a foliation, then  $(X, \mathscr{F})$ is a \emph{foliated space}.
Let $(X, \mathscr{F})$ and $(Y, \mathscr{G})$ be two directed spaces over $S$.
An \emph{$S$-morphism $f : (X, \mathscr{F}) \rightarrow (Y, \mathscr{G})$ of directed spaces} is an $S$-morphism of schemes which is universally separable and admits the existence of the dashed morphism in the diagram
\begin{center}
\begin{tikzcd}
f^* \Omega_{Y/S}^1 \arrow[d] \arrow[r] & \Omega_{X/S}^1 \arrow[d]    \\
f^* \mathscr{G} \arrow[r, dashed] & \mathscr{F}.
\end{tikzcd}
\end{center}
If the dashed morphism in the diagram is an isomorphism, then \emph{$f$ is a weakly \'{e}tale morphism of directed spaces}.
If $\mathscr{F}$ and $\mathscr{G}$ are foliations, then  $f$ is a \emph{morphism of foliated spaces}.
\end{definition}

\begin{remark}[Assumption of universal separability]
An $\mathscr{F}$-invariant morphism is required to be universally separable to avoid some phemonena occuring in positive characteristic.
Let $S = \mathrm{Spec} \, k$ where $k$ is a field of positive characteristic, let $X$ be a scheme over $S$ and let $\mathrm{Frob} : X \rightarrow X$ be the relative Frobenius.
Then for any distribution $\mathscr{F}$ on $X$ over $S$, there exists a factorisation
$$ \Omega_{X/S}^1 \xrightarrow{\sim} \Omega_{\mathrm{Frob}}^1 \rightarrow \mathscr{F}.$$
As a result, it is not true that a subvariety tangent to $\mathscr{F}$ (Definition \ref{def:algebraic_leaves}) is contained in the fibre of an invariant morphism (Proposition \ref{prop:decomposition}).
Whilst this may be seen as pathological beahaviour, it can be exploited to construct a correspondence between foliations and factorisations of the Frobenius morphism (\cite[Lecture III, Proposition 1.9]{MR1468476}).
\end{remark}

\begin{remark}[Terminology of Definition \ref{def:foliated_space}]
\label{rem:terminology}
The term \emph{directed space} can be found in the work of Jean-Pierre Demailly on the Green-Griffiths-Lang conjecture.
For instance, see \cite[\S 0]{MR3445519}.
The term \emph{weakly \'{e}tale morphism of directed spaces} comes from observing that if $f : X \rightarrow Y$ is a weakly \'{e}tale morphism of schemes over $S$, then, as remarked in Lemma \ref{lem:etale_pullback}, the natural morphism $f^* \Omega_{Y/S}^1 \rightarrow \Omega_{X/S}^1$ is an isomorphism.
\end{remark}

\begin{lemma}[Composition of morphisms of directed spaces]
\label{lem:composition_directed}
Let $S$ be a scheme and let $f : (X, \mathscr{F}) \rightarrow (Y, \mathscr{G})$ and $g : (Y, \mathscr{G}) \rightarrow (Z, \mathscr{H})$ be (\emph{respectively} weakly \'{e}tale) morphisms of directed spaces.
Then
$$ g \circ f : (X, \mathscr{F}) \rightarrow (Z, \mathscr{H})$$
is a (\emph{respectively} weakly \'{e}tale) morphism of directed spaces.
\end{lemma}

\begin{proof}
Consider the diagram
\begin{center}
\begin{tikzcd}
g^* f^* \Omega_{Z/S}^1 \arrow[d] \arrow[r] & f^* \Omega_{Y/S}^1 \arrow[d] \arrow[r] & \Omega_{X/S}^1 \arrow[d]    \\
g^* f^* \mathscr{H} \arrow[r, dashed] & f^* \mathscr{G} \arrow[r, dashed] & \mathscr{F}.
\end{tikzcd}
\end{center}
By assumption, the dashed morphisms exist (\emph{respectively} are isomorphisms).
As a result the composition of the dashed morphisms exists (\emph{respectively} is an isomorphism).
To conclude the proof, it is enough to show that the composition of separable field extensions is separable.
Let $\kappa(z) \rightarrow \kappa(y) \rightarrow \kappa(x)$ be separable field extensions and let $\kappa(z) \rightarrow K$ be a field extension.
Using the characterisation in \cite[\href{https://stacks.math.columbia.edu/tag/030W}{Lemma 030W}]{stacks-project}, it is enough to show that $\kappa(x) \otimes_{\kappa(z)} K$ is reduced.
But $\kappa(x) \otimes_{\kappa(z)} K = \kappa(x) \otimes_{\kappa(y)} \left( \kappa(y) \otimes_{\kappa(z)} K \right)$.
By assumption $\kappa(y) \otimes_{\kappa(z)} K$ is a reduced $\kappa(y)$-algebra and the claim follows from \cite[\href{https://stacks.math.columbia.edu/tag/034N}{Lemma 034N}]{stacks-project}.
\end{proof}

\begin{lemma}[Directed spaces and base change]
\label{lem:base_change_foliation}
Let $S$ be a scheme and let $f : (X, \mathscr{F}) \rightarrow (Y, \mathscr{G})$ be a (\emph{respectively} weakly \'{e}tale) morphism of directed spaces over $S$.
Let $h : T \rightarrow S$ be a morphism and define base change morphisms
\begin{center}
\begin{tikzcd}
X^{\prime} \arrow[d, "f^{\prime}"] \arrow[r, "g^{\prime}"] & X \arrow[d, "f"]    \\
Y^{\prime} \arrow[d] \arrow[r, "g"] & Y \arrow[d]    \\
T \arrow[r, "h"] & S.
\end{tikzcd}
\end{center}
Then 
$$f^{\prime} : \left( X^{\prime},  {g^{\prime}}^{*} \mathscr{F} \right) \rightarrow \left( Y^{\prime}, g^{*} \mathscr{G}\right) $$
is a (\emph{respectively} weakly \'{e}tale) morphism of directed spaces.
\end{lemma}

\begin{proof}
Clearly, since $f$ is universally separable, $f^{\prime}$ is universally separable.
By assumption, there is a commutative diagram
\begin{center}
\begin{tikzcd}
f^* \Omega_{Y/S}^1 \arrow[d] \arrow[r] & \Omega_{X/S}^1 \arrow[d]    \\
f^* \mathscr{G} \arrow[r] & \mathscr{F}.
\end{tikzcd}
\end{center}
Applying ${g^{\prime}}^{*}$ and noting that 
\begin{align*}
{g^{\prime}}^{*} \left( f^* \Omega_{Y/S}^1 \right) &= {f^{\prime}}^{*} \left( g^* \Omega_{Y/S}^1 \right) =
{f^{\prime}}^{*} \Omega_{Y^{\prime}/T}^1 &
\text{(\cite[\href{https://stacks.math.columbia.edu/tag/01V0}{Lemma 01V0}]{stacks-project})} \\
{g^{\prime}}^{*} \Omega_{X/S}^1 &=
\Omega_{X^{\prime}/T}^1 &
\text{(\cite[\href{https://stacks.math.columbia.edu/tag/01V0}{Lemma 01V0}]{stacks-project})} \\
{g^{\prime}}^{*} \left( f^* \mathscr{G} \right) &=
{f^{\prime}}^* \left( g^{*} \mathscr{G} \right) &
\end{align*}
yields a commutative diagram
\begin{center}
\begin{tikzcd}
{f^{\prime}}^* \Omega_{Y^{\prime}/T}^1 \arrow[d] \arrow[r] & \Omega_{X^{\prime}/T}^1 \arrow[d]    \\
{f^{\prime}}^* \left( g^{*} \mathscr{G} \right) \arrow[r] & {g^{\prime}}^{*} \mathscr{F}.
\end{tikzcd}
\end{center}
This shows that $f^{\prime}$ is a morphism of directed spaces.
If $f^* \mathscr{G} \rightarrow \mathscr{F}$ is an isomorphism then ${f^{\prime}}^* \left( g^{*} \mathscr{G} \right) \rightarrow {g^{\prime}}^{*} \mathscr{F}$ is also an isomorphism.
\end{proof}

\begin{lemma}[Being a morphism of directed spaces is a stalk-local property]
\label{lem:directed_stalk_local}
Let $S$ be a scheme and let $(X, \mathscr{F})$ and $(Y, \mathscr{G})$ be directed spaces over $S$.
Suppose that $f : X \rightarrow Y$ is an $S$-morphism.
Then $f$ is a (\emph{respectively} weakly \'{e}tale) morphism of directed spaces if and only if, for all $x \in X$, $\kappa(f(x)) \subseteq \kappa(x)$ is universally separable and the dashed morphism in the diagram
\begin{center}
\begin{tikzcd}
\left( f^* \Omega_{Y/S}^1 \right)_x \arrow[d] \arrow[r] & \left(  \Omega_{X/S}^1 \right)_x \arrow[d]    \\
\left( f^* \mathscr{G} \right)_x \arrow[r, dashed] & \mathscr{F}_x.
\end{tikzcd}
\end{center}
exists (\emph{respectively} exists and is an isomorphism).
\end{lemma}

\begin{proof}
Using Remark \ref{rem:char_separability}, it follows that $f : X \rightarrow Y$ is universally separable \par

Let $\mathscr{K} = \mathrm{ker}\left( f^* \Omega_{Y/S}^1 \rightarrow f^* \mathscr{G} \right)$.
$f$ is a morphism of directed spaces if and only if the composition $\phi :\mathscr{K} \rightarrow \Omega_{X/S}^1 \rightarrow \mathscr{F}$ is zero.
This holds if and only if, for all $x \in X$, the composition
\begin{align}
\label{eq:local_composition}
\mathscr{K}_x \rightarrow \Omega_{X/S,x}^1 \rightarrow \mathscr{F}_x
\end{align}
is zero.
But
$$\mathscr{K}_x = \mathrm{ker}\left( \left( f^* \Omega_{Y/S}^1 \right)_x \rightarrow \left( f^* \mathscr{G} \right)_x \right),$$
hence (\ref{eq:local_composition}) holds for all $x \in X$ if and only if the morphism $\left( f^* \mathscr{G} \right)_x \rightarrow \mathscr{F}_x$ exists.
Finally, note that such a morphism is an isomorphism if and only if it is an isomorphism at all the stalks.
\end{proof}

\subsection{Invariant Morphisms}
\label{subsec:invariant}

This subsection defines morphisms invariant by a distribution and provides a criterion to check invariance through a complex of sheaves (Lemma \ref{lem:invariant_complex}).

\begin{definition}[Invariant morphisms]
\label{def:invariant_morphisms}
Let $S$ be a scheme and let $(X, \mathscr{F})$ be a directed space over $S$.
An $S$-morphism \emph{$f : X \rightarrow Y$ is invariant with respect to $\mathscr{F}$} (or \emph{$f : X \rightarrow Y$ is $\mathscr{F}$-invariant}) if 
$$ (X, \mathscr{F}) \rightarrow (Y, 0) $$
is a morphism of directed spaces over $S$.
Equivalently, if there exists a factorisation
$$ \Omega_{X/S}^1 \rightarrow \Omega_{X/Y}^1 \rightarrow \mathscr{F}.$$
\end{definition}

\begin{lemma}[Invariance and base change]
\label{lem:invariant_base_change}
Let $S$ be a scheme and let $(X, \mathscr{F})$ be a directed space over $S$.
Let $Y$ be an $S$-scheme and suppose that $f : X \rightarrow Y$ is an $\mathscr{F}$-invariant $S$-morphism.
Let $h : T \rightarrow S$ be a morphism and let $g^{\prime} : X^{\prime} \rightarrow X$ denote the base change of $h : T \rightarrow S$ by $X \rightarrow S$ and $f^{\prime} : X^{\prime} \rightarrow Y^{\prime}$ the base change of $f : X \rightarrow Y$ by $h$, 
Then $f^{\prime}$ is ${g^{\prime}}^{*} \mathscr{F}$-invariant.
\end{lemma}

\begin{proof}
Observe that $g^*0 = 0$ and apply Lemma \ref{lem:base_change_foliation} to conclude.
\end{proof}

\begin{lemma}[Characterisation of invariance]
\label{lem:invariant_complex}
Let $S$ be a scheme and let $(X, \mathscr{F})$ be a directed space over $S$.
Let $f : X \rightarrow Y$ be a universally separable $S$-morphism and consider the sequences of morphisms
\begin{align}
\label{eq:inv_target}
0 \rightarrow \mathscr{O}_Y \rightarrow f_*\mathscr{O}_X \rightarrow f_* \mathscr{F}, \\
\label{eq:inv_source}
0 \rightarrow f^{-1} \mathscr{O}_Y \rightarrow \mathscr{O}_X \rightarrow \mathscr{F}.
\end{align}
Then
\begin{enumerate}
\item $f$ is $\mathscr{F}$-invariant if and only if (\ref{eq:inv_target}) is a complex.
\item If (\ref{eq:inv_source}) is exact and $f$ is an open morphism with irreducible fibres, then (\ref{eq:inv_target}) is exact.
\end{enumerate}
\end{lemma}

\begin{proof}
\textbf{(1 $\rightarrow$).}
Suppose that $f$ is $\mathscr{F}$-invariant.
Then there exists a factorisation
$$ \Omega_{X/S}^1 \rightarrow \Omega_{X/Y}^1 \rightarrow \mathscr{F}.$$
Now consider the composition
$$ \mathscr{O}_Y \rightarrow f_*\mathscr{O}_X \rightarrow f_* \mathscr{F}.$$
By assumption, there is a factorisation
$$ \mathscr{O}_Y \rightarrow f_*\mathscr{O}_X \rightarrow f_* \Omega_{X/S}^1 \rightarrow f_* \Omega_{X/Y}^1 \rightarrow f_* \mathscr{F},$$
where the composition $\mathscr{O}_Y \rightarrow f_* \Omega_{X/Y}^1$ is zero.

\textbf{(1 $\leftarrow$).}
By assumption, the composition
$$ \mathscr{O}_Y \rightarrow f_*\mathscr{O}_X \rightarrow f_* \Omega_{X/S}^1 \rightarrow f_* \mathscr{F}$$
is zero.
By the universal property of the sheaf of differentials, there exists a sequence of $\mathscr{O}_Y$-modules
$$\Omega_{Y/S}^1 \rightarrow f_* \Omega_{X/S}^1 \rightarrow f_* \mathscr{F}$$
whose composition is zero.
By adjointness of $f_*$ and $f^*$, there exists a sequence of $\mathscr{O}_X$-modules
$$f^* \Omega_{Y/S}^1 \rightarrow \Omega_{X/S}^1 \rightarrow \mathscr{F}$$
whose composition is zero.
Here, the morphism $f^* \Omega_{Y/S}^1 \rightarrow \Omega_{X/S}^1$ is the natural one (\cite[\href{https://stacks.math.columbia.edu/tag/01UV}{Lemma 01UV}]{stacks-project}).
By the universal property of the cokernel, $f$ is $\mathscr{F}$-invariant. \par

\textbf{(2).}
Pushing forward the exact sequence (\ref{eq:inv_source}) yields an exact sequence
$$ 0 \rightarrow f_*f^{-1} \mathscr{O}_Y \rightarrow f_*\mathscr{O}_X \rightarrow f_*\mathscr{F}.$$
Hence, it is enough to show that the adjunction morphism $\mathscr{O}_Y \rightarrow f_* f^{-1} \mathscr{O}_Y$ is an isomorphism.
To this end, it is first shown that
\begin{align}
\label{eq:inverse_sheaf_description}
f^{-1} \mathscr{O}_Y(U) = \mathscr{O}_Y\left( f(U) \right)
\end{align}
for all $U \subseteq X$ open subsets.
Because $f$ is open, $f^{-1} \mathscr{O}_Y$ is the sheaf associated to the presheaf $f^{-1} \mathscr{O}_Y(U) = \mathscr{O}_Y\left( f(U) \right)$ so it is enough to show that this is a sheaf.
Let $U_1$ and $U_2$ be two open subsets of $X$ and let $U_{12}$ be their intersection.
Suppose that $s_1 \in f^{-1} \mathscr{O}_Y(U_1) = \mathscr{O}_Y\left( f(U_1) \right)$ and $s_2 \in f^{-1} \mathscr{O}_Y(U_2) = \mathscr{O}_Y\left( f(U_2) \right)$ are two local sections which agree on the intersection $f^{-1} \mathscr{O}_Y(U_{12}) = \mathscr{O}_Y\left( f(U_{12}) \right)$.
If it can be shown that $f(U_{12}) = f(U_1) \cap f(U_2)$ then, since $\mathscr{O}_Y$ is a sheaf on $Y$, $s_1$ and $s_2$ can be glued to obtain a unique section $s$ over $f(U_1) \cup f(U_2) = f(U_1 \cup U_2)$.
It is therefore shown that 
$$f(U_{12}) = f(U_1) \cap f(U_2).$$
The inclusion $\subseteq$ always holds.
Suppose that $y \in f(U_1) \cap f(U_2)$ and let $x_1 \in U_1$ and $x_2 \in U_2$ be such that $f(x_1) = f(x_2) = y$.
Consider the fibre $X_y = X \times_Y \mathrm{Spec}\, \kappa(y)$.
Note that $U_1 \cap X_y \ni x_1$ and $U_2 \cap X_y \ni x_2$ are two non-empty open subsets of $X_y$.
By assumption $X_y$ is irreducible, hence the intersection $U_{12} \cap X_y$ is non-empty.
Let $x$ be in this intersection, then $x \in U_{12}$ and $f(x) = y$.
This shows (\ref{eq:inverse_sheaf_description}).
Finally, let $V \subseteq Y$ be an open subset of $Y$.
Then the morphism
$$ \mathscr{O}_Y(V) \rightarrow f_* f^{-1} \mathscr{O}_Y(V) = \mathscr{O}_Y \left( f \left( f^{-1}(V) \right) \right) $$
is an isomorphism.
Indeed, since $f$ is surjective, $f \left( f^{-1}(V) \right) = V$.
\end{proof}

\subsection{Algebraic Leaves}
\label{subsec:algebraic_leaves}

This subsection defines tangent subvarieties, algebraic leaves and proves that a tangent subvariety is contained in the fibre of an invariant morphism (Proposition \ref{prop:decomposition}).
Furthermore, it introduces the functor of smooth algebraic leaves and shows a uniqueness-type result (Lemma \ref{lem:unique_algebraic_leaf}).

\begin{definition}[Smooth algebraic leaves]
\label{def:algebraic_leaves}
Let $S$ be a scheme and let $(X, \mathscr{F})$ be a directed space over $S$. Let $T$ be an $S$-scheme. A \emph{smooth family of subvarieties tangent to $\mathscr{F}$ over $T$} (\emph{\emph{respectively} smooth family of algebraic leaves of $ \mathscr{F}$ over $T$}) is an $S$-scheme $Z$ together with an $S$-morphism $Z \rightarrow X \times_S T$ such that
\begin{enumerate}[label={(\roman*)}]
\item \label{tangent:immersion}
$Z \rightarrow X \times_S T$ is a closed immersion.
\item \label{tangent:smooth}
$Z$ is smooth over $T$.
\item \label{tangent:integral}
$Z \rightarrow T$ has geometrically integral fibres.
\item \label{tangent:distribution}
$ \left( Z, \Omega_{Z/T}^1 \right) \rightarrow (X, \mathscr{F}) $ is a (\emph{respectively} weakly \'{e}tale) morphism of directed spaces.
\end{enumerate}
\end{definition}

\begin{lemma}[Smooth algebraic leaves and open subsets]
\label{lem:algebraic_leaves_open_subsets}
Let $S$ be a scheme and let $(X, \mathscr{F})$ be a directed space over $S$.
Let $T$ be a scheme over $S$ and let $g : U \rightarrow X$ be an open immersion.
Suppose that $\iota : Z \rightarrow X \times_S T$ is a smooth family of subvarieties tangent to $\mathscr{F}$ (respectively a smooth family of algebraic leaves of $\mathscr{F}$) over $T$ and let $W = Z \times_X U$.
Suppose that $W \rightarrow T$ is surjective, then $W$ is a smooth family of subvarieties tangent to $g^{\#} \mathscr{F}$ (\emph{respectively} a smooth family of algebraic leaves of $g^{\#} \mathscr{F}$)  over $T$.
\end{lemma}

\begin{proof}
Because closed immersions are stable under base change (\cite[\href{https://stacks.math.columbia.edu/tag/01JY}{Lemma 01JY}]{stacks-project}), $W = Z \times_X U \rightarrow U \times_S T$ is a closed immersion.
Hence \ref{tangent:immersion} holds.
Further $W \rightarrow T$ is smooth since the morphisms in $W \rightarrow Z \rightarrow T$ are smooth.
The former because it is the base change of an open immersion (\cite[\href{https://stacks.math.columbia.edu/tag/01VC}{Lemma 01VC}]{stacks-project}), the latter by assumption.
Thus \ref{tangent:smooth} holds.
Having geometrically integral fibres follows from the fact that $W \rightarrow Z$ is an open immersion, hence the geometric fibres of $W \rightarrow T$ are non-empty (since $W \rightarrow T$ is surjective) open subsets of the geometric fibres of $Z \rightarrow T$.
But non-empty open subschemes of integral schemes are integral.
Thus \ref{tangent:integral} holds.
Furthermore
$$\left( W, \Omega_{W/T}^1 \right) \rightarrow (U, g^{\#} \mathscr{F})$$
is a (\emph{respectively} weakly \'{e}tale) morphism of directed spaces.
This follows from Lemma \ref{lem:base_change_foliation} and Lemma \ref{lem:etale_pullback}.
Hence \ref{tangent:distribution} also holds.
\end{proof}

\begin{lemma}[Smooth algebraic leaves and base change]
\label{lem:base_change_algebraic_leaves}
Let $S$ be a scheme and let $(X, \mathscr{F})$ be a directed space over $S$.
Let $T$ be a scheme over $S$ and let $T^{\prime} \rightarrow T$ be a morphism.
Suppose that $\iota : Z \rightarrow X \times_S T$ is a smooth family of subvarieties tangent to $\mathscr{F}$ (respectively a smooth family of algebraic leaves of $\mathscr{F}$) over $T$.
Then the base change 
$$ \iota^{\prime} : Z^{\prime} = Z \times_T T^{\prime} \rightarrow X \times_S T^{\prime} $$
is a smooth family of subvarieties tangent to $\mathscr{F}$ (respectively a smooth family of algebraic leaves of $\mathscr{F}$) over $T^{\prime}$.
\end{lemma}

\begin{proof}
\ref{tangent:immersion}, \ref{tangent:smooth} and \ref{tangent:integral} are stable under base change.
Note that, by \cite[\href{https://stacks.math.columbia.edu/tag/01V0}{Lemma 01V0}]{stacks-project},
$$ \left(Z^{\prime}, \Omega_{Z^{\prime}/T^{\prime}}^1 \right) \rightarrow \left(Z, \Omega_{Z/T}^1 \right) $$
is a weakly \'{e}tale morphism of directed spaces, hence
$$ \left(Z^{\prime}, \Omega_{Z^{\prime}/T^{\prime}}^1 \right) \rightarrow \left(X, \mathscr{F} \right) $$
is a (\emph{respectively} weakly \'{e}tale) morphism of directed spaces by Lemma \ref{lem:composition_directed}.
\end{proof}

\begin{proposition}[Tangent variety is contained in fibre of invariant morphism]
\label{prop:decomposition}
Let $S$ be a scheme and let $(X, \mathscr{F})$ be a directed space over $S$.
Suppose that $f : X \rightarrow Y$ is an $\mathscr{F}$-invariant morphism of schemes. Let $\mathrm{Spec} \, L \rightarrow S$ be a morphism of schemes where $L$ is an algebraically closed field and let $Z \rightarrow X \times_S \mathrm{Spec} \, L$ be a smooth subvariety tangent to $\mathscr{F}$ over $\mathrm{Spec}\,L$.
Then there exists a unique factorisation 
\begin{center}
\begin{tikzcd}
Z \arrow[d] \arrow[r] & X \arrow[d]  \\
\mathrm{Spec}\,L \arrow[r, dashed] & Y.
\end{tikzcd}
\end{center}
\end{proposition}

\begin{proof}
Since $Z \rightarrow \mathrm{Spec}\,L$ is an epimorphism, uniqueness is clear. \par

After base changing by $\mathrm{Spec} \, L \rightarrow S$, it may be assumed that $f : X \rightarrow Y$ is an $\mathscr{F}$-invariant morphism of schemes over $\mathrm{Spec}\,L$ (Lemma \ref{lem:invariant_base_change}).
The composition of morphisms of directed spaces
$$ \left( Z, \Omega_{Z/\mathrm{Spec}\,L}^1 \right) \xrightarrow{g} (X, \mathscr{F}) \xrightarrow{f} (Y,0)$$
is a morphism of directed spaces (Lemma \ref{lem:composition_directed}).
It follows that
$$ \Omega_{Z/Y}^1 \cong \Omega_{Z/\mathrm{Spec}\,L}^1.$$
Let $z$ be the generic point of $Z$ and let $y$ be its image under $f \circ g$. Because $g \circ f$ is residually separable and $L$ is perfect, there are separable field extensions
$$ L \rightarrow \kappa(y) \rightarrow \kappa(z).$$
Because $Z$ is locally of finite type over $\mathrm{Spec}\,L$, $ L \rightarrow \kappa(z)$ is finitely generated.
But then $L \rightarrow \kappa(y)$ and $\kappa(y) \rightarrow \kappa(z)$ are also finitely generated, hence separably generated.
Now
\begin{align*}
\mathrm{tr\,deg}_{\kappa(y)} \kappa(z) &=
\mathrm{dim}_{\kappa(z)} \Omega_{\kappa(z)/\kappa(y)}^1 & \\
&= \mathrm{dim}_{\kappa(z)} \Omega_{Z/Y}^1 \otimes \kappa(z) & \\
&= \mathrm{dim}_{\kappa(z)} \Omega_{Z/L}^1 \otimes \kappa(z) & \\
&= \mathrm{dim}_{\kappa(z)} \Omega_{\kappa(z)/L}^1 & \\
&= \mathrm{tr\,deg}_{L} \kappa(z) & .
\end{align*}

This implies that $\mathrm{tr\,deg}_{L} \kappa(y) = 0$ so that $L \rightarrow \kappa(y)$ is an algebraic field extension.
Because $L$ is algebraicailly closed $L = \kappa(y)$.
Now $y$ is an $L$-rational point of $Y$, hence it is a closed point of $Y$.
Since $z^{\prime}$ is a specialisation of $z$, $y^{\prime}$ is a specialisation of $y$.
But $y$ is a closed point, hence $y^{\prime} = y$.
As a result, the set-theoretic image of $Z \rightarrow Y$ is $\{ y \}$. 
Using the fact that $Z$ is reduced and applying \cite[\href{https://stacks.math.columbia.edu/tag/056B}{Lemma 056B}]{stacks-project} yields that $\mathrm{Spec}\,L \rightarrow Y$ is the scheme-theoretic image of $Z \rightarrow Y$, hence there is a factorisation
$$ Z \rightarrow \mathrm{Spec}\,L \rightarrow Y.$$
\end{proof}

\begin{remark}[Construction of the moduli space]
\label{rem:moduli_space}
Proposition \ref{prop:decomposition} is the key step in showing that a geometric quotient is a moduli space for the geometric points.
If one wanted to construct the moduli space parametrising all algebraic leaves, one would need to show the same statement as above for an arbitrary family of algebraic leaves $Z \rightarrow X \times_S T$ over an $S$-scheme $T$.
However, this statement seems difficult to prove.
\end{remark}

\begin{definition}[Functor of smooth algebraic leaves]
\label{def:functor_algebraic_leaves}
Let $S$ be a scheme and let $(X, \mathscr{F})$ be a directed space over $S$.
Let $T$ be a scheme over $S$.
The \emph{functor of smooth algebraic leaves} is defined to be
$$ F(T) = \left\{ Z \rightarrow X \times_S T \, | \, \text{$Z$ is a smooth family of algebraic leaves of $\mathscr{F}$ over $T$} \right\} $$
This is well-defined since the pullback of a smooth family of algebraic leaves is a smooth family of algebraic leaves (Lemma \ref{lem:base_change_algebraic_leaves}).
A morphism $\pi : X \rightarrow Y$ represents the functor $F$ if for all $S$ schemes $T$, there is a natural bijection
$$ \mathrm{Hom} (T, Y) \longleftrightarrow F(T) $$ 
given by the map which takes an $S$-morphism $T \rightarrow Y$ to the fibre product $X \times_Y T$.
If such a morphism $\pi$ exists, then, under the bijection, $\pi : X \rightarrow Y$ corresponds to the \emph{algebraic graph} of the distribution.
\end{definition}

\begin{lemma}[Uniqueness of smooth algebraic leaves]
\label{lem:unique_algebraic_leaf}
Let $S$ be a scheme and let $(X, \mathscr{F})$ be a directed space over $S$.
Let $T$ be a scheme over $S$.
Suppose that $\iota : Z \rightarrow X \times_S T$ and $\iota^{\prime} : Z^{\prime} \rightarrow X \times_S T$ are two smooth algebraic leaves of $\mathscr{F}$ over $T$ such that there exists $T$-morphism $\gamma : Z \rightarrow Z^{\prime}$ admitting a factorisation $\iota = \iota^{\prime} \circ \gamma$.
Then $\gamma : Z \rightarrow Z^{\prime}$ is an isomorphism.
\end{lemma}

\begin{proof}
Since $\iota^{\prime} : Z^{\prime} \rightarrow X \times_S T$ is an immersion, it is separated (\cite[\href{https://stacks.math.columbia.edu/tag/01L7}{Lemma 01L7}]{stacks-project}).
It follows that $\gamma : Z \rightarrow Z^{\prime}$ is a closed immersion (\cite[\href{https://stacks.math.columbia.edu/tag/07RK}{Lemma 07RK}]{stacks-project}). \par

Now consider the six term exact sequence given by the na\"{i}ve cotangent complexes of the morphisms $Z \rightarrow Z^{\prime} \rightarrow T$ (\cite[\href{https://stacks.math.columbia.edu/tag/0E44}{Lemma 0E44}]{stacks-project}) 
$$ H^{-1} \left( {NL}_{Z/T} \right) \rightarrow H^{-1} \left( {NL}_{Z/Z^{\prime}} \right) \rightarrow \gamma^* \Omega_{Z^{\prime}/T}^1 \rightarrow \Omega_{Z/T}^1 \rightarrow \rightarrow \Omega_{Z/Z^{\prime}}^1 \rightarrow 0. $$
Since $Z \rightarrow T$ is smooth, $H^{-1} \left( {NL}_{Z/T} \right) = 0$.
Because both $\left(Z, \Omega_{Z/T}^1 \right) \rightarrow \left(X, \mathscr{F} \right)$ and $\left(Z^{\prime}, \Omega_{Z^{\prime}/T}^1 \right) \rightarrow \left(X, \mathscr{F} \right)$ are weakly \'{e}tale morphisms of directed spaces, $\gamma^* \Omega_{Z^{\prime}/T}^1 \rightarrow \Omega_{Z/T}^1$ is an isomorphism.
It follows that both $\Omega_{Z/Z^{\prime}}^1  = 0$ and $H^{-1} \left( {NL}_{Z/Z^{\prime}} \right) = 0$.
In other words, the na\"{i}ve cotangent complex of $\gamma$ is exact.
Furthermore $\gamma$ is locally of finite presentation (\cite[\href{https://stacks.math.columbia.edu/tag/02FV}{Lemma 02FV}]{stacks-project}) so, by \cite[\href{https://stacks.math.columbia.edu/tag/0G7Z}{Lemma 0G7Z}]{stacks-project}, it is \'{e}tale, hence flat.
$\gamma$ is also a monomorphism (\cite[\href{https://stacks.math.columbia.edu/tag/01L7}{Lemma 01L7}]{stacks-project}), therefore \cite[\href{https://stacks.math.columbia.edu/tag/025G}{Theorem 025G}]{stacks-project} implies that $\gamma$ is an open immersion.
It follows that $\gamma$ is an open and closed immersion. \par

Now, the fibres of $Z \rightarrow T$ are non-empty (since $Z \rightarrow T$ is surjective) connected components of the fibres of $Z^{\prime} \rightarrow T$.
As these are geometrically integral, $\gamma$ induces an isomorphism between the fibres of $Z \rightarrow T$ and the fibres of $Z^{\prime} \rightarrow T$. 
This shows that $\gamma$ is surjective.
As remarked before, $\gamma$ is also flat closed immersion hence, by \cite[\href{https://stacks.math.columbia.edu/tag/04PW}{Lemma 04PW}]{stacks-project}, it is an isomorphism.
\end{proof}

\begin{lemma}[Functor of smooth geometric leaves]
\label{lem:functor_geometric_leaves}
Let $S$ be a scheme and let $(X, \mathscr{F})$ be a directed space over $S$.
Suppose that there exists an $S$-morphism $\pi : X \rightarrow Y$ which is $\mathscr{F}$-invariant and has the following property: for all geometric points $\mathrm{Spec}\, L \rightarrow Y$, where $L$ is an algebraically closed field, the fibre product $X \times_Y \mathrm{Spec}\,L$ is a smooth algebraic leaf of $\mathscr{F}$ over $\mathrm{Spec}\,L$.
Then $\pi$ represents geometric points of the functor of smooth algebraic leaves. 
\end{lemma}

\begin{proof}
Let $\alpha$ be the map of sets
$$ \mathrm{Hom} (\mathrm{Spec}\,L, Y) \xrightarrow{\alpha} F(\mathrm{Spec}\,L) $$ 
which takes a morphism $\mathrm{Spec}\,L \rightarrow Y$ and returns $X \times_Y \mathrm{Spec}\,L$, a smooth algebraic leaf of $\mathscr{F}$ over $\mathrm{Spec}\,L$.
Define an inverse map $\beta$ which takes a smooth algebraic leaf $Z$ of $\mathscr{F}$ over $\mathrm{Spec}\,L$ and returns the unique morphism $\mathrm{Spec}\,L \rightarrow Y$ given by Proposition \ref{prop:decomposition}. 
It is shown that $\alpha$ and $\beta$ are inverse to each other.
By uniqueness of the morphism $\mathrm{Spec}\,L \rightarrow Y$, it is clear that $\beta \circ \alpha = \mathds{1}$.
Conversely let $Z$ be a smooth algebraic leaf of $\mathscr{F}$ over $\mathrm{Spec}\,L$ and let $Z^{\prime} = \alpha \left( \beta (Z) \right)$.
Then $Z^{\prime}$ is a smooth algebraic leaf of $\mathscr{F}$ over $T$ admitting a $\mathrm{Spec}\,L$-factorisation $Z \rightarrow Z^{\prime} \rightarrow X \times_S \mathrm{Spec}\,L$.
By Lemma \ref{lem:functor_geometric_leaves}, $Z = Z^{\prime}$.
\end{proof}

\section{Quotient Spaces}
\label{sec:quotient_spaces}

This section introduces two different notions of quotients and relates them. 
\S \ref{subsec:categorical_quotients} introduces categorical quotients.
\S \ref{subsec:geometric_quotient} introduces geometric quotients  and shows that geometric quotients are categorical.
\S \ref{subsec:glueing} shows that geometric quotients can be glued together.

\subsection{Categorical Quotients}
\label{subsec:categorical_quotients}

This subsection defines categorical quotients.

\begin{definition}[Categorical quotients]
\label{def:categorical_quotient}
Let $S$ be a scheme and let $(X, \mathscr{F})$ be a directed space over $S$.
A \emph{categorical quotient} of $X$ by $\mathscr{F}$ (or simply \emph{categorical quotient}) is an $S$-morphism $\pi : X \rightarrow Y$ such that
\begin{enumerate}
\item $\pi$ is $\mathscr{F}$-invariant.
\item Any $\mathscr{F}$-invariant $S$-morphism uniquely factors through $\pi$.
That is, if ${\pi^{\prime}} : X \rightarrow {Y^{\prime}}$ is an $\mathscr{F}$-invariant $S$-morphism, then there exists a unique $S$-morphism $f : Y \rightarrow {Y^{\prime}}$ such that ${\pi^{\prime}} = f \circ \pi$.
\end{enumerate}
\end{definition}

\begin{lemma}[Uniqueness of categorical quotients]
\label{lem:unique_categorical_quotient}
Let $S$ be a scheme and let $(X, \mathscr{F})$ be a foliated space over $S$. If a categorical quotient $\pi : X \rightarrow Y$ exists, then it is unique.
\end{lemma}

\begin{proof}
Clear.
\end{proof}

\subsection{Geometric Quotients}
\label{subsec:geometric_quotient}

This subsection introduces geometric quotients and proves that geometric quotients are categorical quotients. The proof is essentially identical to \cite[Chapter 0, \S 2, Proposition 0.1]{MR1304906}.

\begin{definition}[Geometric quotients]
\label{def:geom_quotients}
Let $S$ be a scheme and let $(X, \mathscr{F})$ be a directed space over $S$.
A \emph{geometric quotient} of $X$ by $\mathscr{F}$ (or simply \emph{geometric quotient}) is an $S$-morphism $\pi : X \rightarrow Y$ to an $S$-scheme $Y$ such that
\begin{enumerate}[label=(\arabic*)]
\item \label{geometric:invariant} It is $\mathscr{F}$-invariant.
\item \label{geometric:functor} It represents geometric points of the functor of smooth algebraic leaves.
\item \label{geometric:open} It is universally open.
\item \label{geometric:exact} The sequence of abelian sheaves on $X$
$$ 0 \rightarrow \pi^{-1}\mathscr{O}_Y \rightarrow \mathscr{O}_X \rightarrow \mathscr{F} $$
is exact.
\end{enumerate}
\end{definition}

\begin{proposition}[Geometric quotients are categorical]
\label{prop:geometric_categorical_quotient}
Let $S$ be a scheme and let $(X, \mathscr{F})$ be a directed space over $S$. Suppose that a geometric quotient $\pi : X \rightarrow Y$ exists. Then it is a categorical quotient.
\end{proposition}

\begin{proof}
By \ref{geometric:invariant}, $\pi : X \rightarrow Y$ is $\mathscr{F}$-invariant. \par

Now suppose that ${\pi^{\prime}} : X \rightarrow Y^{\prime}$ is an $\mathscr{F}$-invariant $S$-morphism.
Define $f : Y \rightarrow Y^{\prime}$ as a map of sets as follows: for $y \in Y$, let $x \in X_y = X \times_Y \mathrm{Spec}\,\kappa(y)$ and let $f(y) = \pi(x)$. 
It is shown that this map is well-defined and unique. 
Let $X_L = X \times_Y \mathrm{Spec}\,L$ where $L$ is the algebraic closure of $\kappa(y)$. 
By Property \ref{geometric:functor}, $X_L$ is a smooth algebraic leaf of $\mathscr{F}$ over $\mathrm{Spec}\,L$.
But by Proposition \ref{prop:decomposition}, the image of the composition
$$ X_L \rightarrow X \rightarrow Y^{\prime} $$
consists of a single point: ${\pi^{\prime}}(x)$.
This shows that $f$ is well-defined.
It is clear that $f$ is unique since $\pi$ is an epimorphism in the category of topological spaces. \par

Next, it is shown that $f$ is continuous. If $V \subseteq Y^{\prime}$ is an open subset, it is easy to see that
$$ f^{-1}(V) = \pi\left( {\pi^{\prime}}^{-1}(V) \right).$$
This is open by Property \ref{geometric:open}. \par

Finally, it is shown that $f$ admits a unique scheme-theoretic structure.
Let $V \subseteq Y^{\prime}$ be an open set, a pullback map
$$ \mathscr{O}_{Y^{\prime}}(V) \rightarrow \mathscr{O}_Y \left(f^{-1}(V) \right) $$
is constructed.
Because ${\pi^{\prime}}$ is $\mathscr{F}$-invariant, by Lemma \ref{lem:invariant_complex}.(1), the sequence
$$ 0 \rightarrow \mathscr{O}_{Y^{\prime}}(V) \rightarrow \mathscr{O}_X \left({\pi^{\prime}}^{-1}(V) \right) \rightarrow \mathscr{F} \left({\pi^{\prime}}^{-1}(V) \right)$$
is a complex of abelian groups.
Hence the sequence
$$ 0 \rightarrow \mathscr{O}_{Y^{\prime}}(V) \rightarrow \mathscr{O}_X \left(\pi^{-1}f^{-1}(V) \right) \rightarrow \mathscr{F} \left(\pi^{-1}f^{-1}(V) \right)$$
is a complex of abelian groups.
But by Property \ref{geometric:functor}, $f$ has irreducible fibres and by Property \ref{geometric:open}, it is an open morphism.
Hence Property \ref{geometric:exact} and Lemma \ref{lem:invariant_complex}.(2) imply that
$$ \mathscr{O}_Y \left(f^{-1}(V)\right) = \mathrm{ker} \left( \mathscr{O}_X \left(\pi^{-1}f^{-1}(V) \right) \rightarrow \mathscr{F} \left(\pi^{-1}f^{-1}(V) \right) \right).$$
By the universal property of the kernel, there exists a unique morphism of abelian groups
\begin{align}
\label{eq:scheme_structure}
\mathscr{O}_{Y^{\prime}}(V) \rightarrow \mathscr{O}_Y \left(f^{-1}(V) \right).
\end{align}
By Lemma \ref{lem:morphism_groups_rings}, this is a morphism of rings.
\end{proof}

\begin{lemma}[Morphism of rings]
\label{lem:morphism_groups_rings}
Let $A$, $B$ and $C$ be rings.
Suppose that there are morphism of abelian groups $C\rightarrow A$ and $A \rightarrow B$ such that $A \rightarrow B$ is an injective morphism of rings and the composition $C \rightarrow B$ is a morphism of rings.
Then $C \rightarrow A$ is a morphism of rings.
\end{lemma}

\begin{proof}
Clear.
\end{proof}

\subsection{Glueing Geometric Quotients}
\label{subsec:glueing}

This subsection shows that it is enough to construct a geometric quotient locally. The irreducibility of the fibres is key in the arguments.

\begin{lemma}[Restriction of geometric quotients I]
\label{lem:open_geometric_quotient}
Let $S$ be a scheme and let $(X, \mathscr{F})$ be a directed space over $S$.
Suppose that a geometric quotient $\pi : X \rightarrow Y$ for $\mathscr{F}$ exists.
Let $\iota_U : U \rightarrow X$ be an open immersion and let $V \subseteq Y$ denote the set-theoretic image of $U$ by $\pi \circ \iota_U$.
Then $\iota_V : V \rightarrow Y$ is an open subset and $\pi|_U : U \rightarrow V$ is a geometric quotient of $(U, \iota_U^{\#}\mathscr{F})$.
\end{lemma}

\begin{proof}
Note that $\iota_V : V \subseteq Y$ is open because $\pi$ is open (Property \ref{geometric:open}). \par

Property \ref{geometric:invariant} is preserved since being a morphism of directed spaces is a local property (Lemma \ref{lem:directed_stalk_local}). \par

Property \ref{geometric:functor} is preserved.
Let $L$ be an algebraically closed field and let $\mathrm{Spec}\,L \rightarrow V$ be a morphism.
By Lemma \ref{lem:functor_geometric_leaves}, it is enough to show that $Z = U \times_V \mathrm{Spec}\,L$ is a smooth algebraic leaf of $\iota_U^{\#} \mathscr{F}$ over $L$.
Note that, since $V \subseteq Y$ is an open immersion, $Z = U \times_Y \mathrm{Spec}\,L$.
Since $Z$ is, by construction, non-empty, the hypothesis of Lemma \ref{lem:algebraic_leaves_open_subsets} are satisifed and the claim follows. \par

Property \ref{geometric:open} is clearly preserved. \par

Property \ref{geometric:exact} is also preserved. Indeed, pulling back the exact sequence by the exact functor $\iota_U^{-1}$ yields
$$ 0 \rightarrow \iota_U^{-1} \pi^{-1}\mathscr{O}_Y \rightarrow \iota_U^{-1} \mathscr{O}_X \rightarrow \iota_U^{-1} \mathscr{F}. $$
But $\iota_U^{-1} \circ \pi^{-1} = \pi_U^{-1} \circ \iota_V^{-1}$. Further, because $\iota_U$ and $\iota_V$ are open immersions, $\iota_U^{-1} = \iota_U^*$ and $\iota_V^{-1} = \iota_V^*$. This gives exactness of
$$ 0 \rightarrow \pi_U^{-1} \mathscr{O}_V \rightarrow \mathscr{O}_U \rightarrow \iota_U^{*} \mathscr{F}. $$
Finally, by Lemma \ref{lem:etale_pullback}, $\iota_U^{*} \mathscr{F} = \iota_U^{\#} \mathscr{F}$.
\end{proof}

\begin{lemma}[Restriction of geometric quotients II]
\label{lem:intersection_geometric_quotient}
Let $S$ be a scheme and let $(X, \mathscr{F})$ be a directed space over $S$.
Suppose that a geometric quotient $\pi : X \rightarrow Y$ for $\mathscr{F}$ exists.
Let
\begin{center}
\begin{tikzcd}
U_1 \cap U_2 \arrow[r] \arrow[d] & U_1 \arrow[d] \\
U_2 \arrow[r] & X.
\end{tikzcd}
\end{center}
be a Cartesian diagram of open immersions.
Then there exists a commutative diagram
\begin{center}
\begin{tikzcd}[row sep=scriptsize, column sep=scriptsize]
& U_1 \arrow[rr] \arrow[dd] & & X \arrow[dd] \\
U_1 \cap U_2 \arrow[ur] \arrow[rr, crossing over] \arrow[dd] & & U_2 \arrow[ur] \\
& V_1 \arrow[rr] & & Y \\
V_{12} \arrow[ur] \arrow[rr] & & V_2 \arrow[from=uu, crossing over] \arrow[ur]
\end{tikzcd}
\end{center}
where the vertical arrows are geometric quotients and the bottom face of the cube is Cartesian.
\end{lemma}

\begin{proof}
Applying Lemma \ref{lem:open_geometric_quotient} gives geometric quotients $U_1 \rightarrow V_1$, $U_2 \rightarrow V_2$ and $U_1 \cap U_2 \rightarrow V_{12}$ together with unique open immersions 
\begin{center}
\begin{tikzcd}
V_{12} \arrow[r] \arrow[d] & V_1 \arrow[d] \\
V_2 \arrow[r] & Y
\end{tikzcd}
\end{center}
forming a commutative diagram.
It is shown that $V_{12} = V_1 \times_Y V_2$.
Because all the morphisms involved in the fibre product are open immersions, it is enough to show that, if $y \in V_1 \cap V_2$, then $y \in V_{12}$.
Let $X_y$ be the fibre of $\pi$ over $y$ and recall that, by Property \ref{geometric:functor} of a geometric quotient, it is irreducible. Consider the open subsets $X_y \cap U_1$ and $X_y \cap U_2$. Because $y \in V_1 = \pi(U_1)$ and $y \in V_2 = \pi(U_2)$, these are non-empty open subsets. By irreducibility, there exists an $x \in X_y \cap U_1 \cap U_2$. But now, $\pi(x) = y$ and $x \in U_1 \cap U_2$ which shows that $y \in V_{12}$.
\end{proof}

\begin{proposition}[Glueing geometric quotients]
\label{prop:glue_geometric_quotient}
Let $S$ be a scheme and let $(X, \mathscr{F})$ be a directed space over $S$.
Let $I$ be an indexing set and let $\{ \iota_i : U_i \rightarrow X \}_{i \in I}$ be an affine open cover of $X$.
Suppose that, for all $i \in I$, a geometric quotient $\pi_i : U_i \rightarrow V_i$ for $\iota_i^{\#} \mathscr{F}$ exists.
Then there exists an $S$ scheme $Y$ and a geometric quotient $\pi : X \rightarrow Y$ for $\mathscr{F}$.
\end{proposition}

\begin{proof}
Firstly, the $S$-scheme $Y$ and the $S$-morphism $\pi : X \rightarrow Y$ are constructed. \par

\cite[\href{https://stacks.math.columbia.edu/tag/01JC}{Lemma 01JC}]{stacks-project} is used to glue the schemes $V_i$, for $i \in I$ to obtain $Y$.
If $i$ and $j$ are two indices in $I$, define an open subset $V_{ij}$ of $V_i$ by applying Lemma \ref{lem:open_geometric_quotient} to the open immersion $U_i \cap U_j \rightarrow U_i$ and get a geometric quotient $U_i \cap U_j \rightarrow V_{ij}$.
Note the asymmetry in the notation.
Now, because geometric quotients are categorical quotients (Proposition \ref{prop:geometric_categorical_quotient}), and the morphism $U_i \cap U_j \rightarrow V_{ji}$ is, in particular, $\mathscr{F}$-invariant, there exists a unique morphism $\varphi_{ij} : V_{ij} \rightarrow V_{ji}$.
Note that $V_{ii} = V_i$ and $\varphi_{ii} = \mathds{1}_{V_i}$.
In order to glue these schemes, it is necessary to show that two conditions are satisfied.
For any, not necessarily distinct, triple of indices $i$, $j$ and $k$ in $X$
\begin{enumerate}
\item $\varphi_{ij}^{-1}\left(V_{ji} \cap V_{jk} \right) = V_{ij} \cap V_{ik}$.
\item The diagram
\begin{center}
\begin{tikzcd}
V_{ij} \cap V_{ik} \arrow[rr, "\varphi_{ik}"] \arrow[rd, "\varphi_{ij}"'] & & V_{ki} \cap V_{kj} \\
& V_{ji} \cap V_{jk} \arrow[ru, "\varphi_{jk}"'] &                                                          
\end{tikzcd}
\end{center}
is commutative.
\end{enumerate}

Because $V_{ij}$ and $V_{ji}$ are both geometric quotients of $U_i \cap U_j$, $\varphi_{ji} = \varphi_{ij}^{-1}$.
Now, applying Lemma \ref{lem:intersection_geometric_quotient} to the diagram
\begin{center}
\begin{tikzcd}
U_i \cap U_j \cap U_k \arrow[r] \arrow[d] & U_i \cap U_j \arrow[d] \\
U_i \cap U_k \arrow[r] & U_i
\end{tikzcd}
\end{center}
gives that $V_{ij} \cap V_{ik} \subseteq V_i$ is the geometric quotient of $U_i \cap U_j \cap U_k$.
Repeating the argument, swapping $i$ and $j$, yields that $V_{ji} \cap V_{jk} \subseteq V_j$ is also the geometric quotient of $U_i \cap U_j \cap U_k$.
In particular, there exists an isomorphism $\phi : V_{ji} \cap V_{jk} \rightarrow V_{ij} \cap V_{ik}$.
Now the resulting diagram
\begin{center}
\begin{tikzcd}
V_{ji} \arrow[r, "\varphi_{ji}"] & V_{ij} \\
V_{ji} \cap V_{jk} \arrow[u] \arrow[r, "\phi"'] & V_{ij} \cap V_{ik} \arrow[u]
\end{tikzcd}
\end{center}
is commutative.
Indeed, since the morphism $U_i \cap U_j \cap U_k \rightarrow U_i \cap U_j \rightarrow V_{ij}$ is $\mathscr{F}$-invariant, there can be at most one morphism from the categorical quotient $V_{ji} \cap V_{jk} \rightarrow V_{ij}$.
This shows that $\phi = \varphi_{ji}|_{V_{ji} \cap V_{jk}}$.
Now, (1) follows easily since
$$ \varphi_{ij}^{-1}\left(V_{ji} \cap V_{jk} \right) =
\varphi_{ji}\left(V_{ji} \cap V_{jk} \right) =
\phi \left(V_{ji} \cap V_{jk} \right) =
V_{ij} \cap V_{ik}.$$
On the other hand, (2) also follows.
Indeed, since the morphism $U_i \cap U_j \cap U_k \rightarrow V_{ki} \cap V_{kj}$ is $\mathscr{F}$-invariant, there can be at most one morphism from the categorical quotient $V_{ij} \cap V_{ik} \rightarrow V_{ki} \cap V_{kj}$.
Hence
$$ \varphi_{ik} = \varphi_{jk} \circ \varphi_{ij}.$$
Let $Y$ denote the scheme obtained by performing the glueing. \par

Next, it is shown that the local morphisms $\pi_i : U_i \rightarrow V_i \rightarrow Y$ can be glued in order to obtain a morphism $\pi : X \rightarrow Y$. 
By \cite[\href{https://stacks.math.columbia.edu/tag/01JB}{Lemma 01JB}]{stacks-project}, to give a morphism $\pi : X \rightarrow Y$ is equivalent to giving morphisms $\pi_i : U_i \rightarrow V_i \rightarrow Y$ indexed by $I$, such that $\pi_j|_{U_i \cap U_j} = \varphi_{ij} \circ \pi_i|_{U_i \cap U_j}$.
This is true by construction. Indeed $\varphi_{ij} : V_{ij} \rightarrow V_{ji}$ is defined to be the unique morphism such that $\pi_j|_{U_i \cap U_j} = \varphi_{ij} \circ \pi_i|_{U_i \cap U_j}$.
Therefore there exists a morphism $\pi : X \rightarrow Y$ restricting to morphisms $\pi_i : U_i \rightarrow V_i \rightarrow Y$ for all $i \in I$. \par

Finally, it is shown that $\pi : X \rightarrow Y$ is a geometric quotient. \par

Property \ref{geometric:invariant} holds since being $\mathscr{F}$-invariant is a local property (Lemma \ref{lem:directed_stalk_local}). \par

Property \ref{geometric:open} holds.
Indeed, being open is a local property. \par

Property \ref{geometric:exact} also holds.
To see this, it is enough to show that, for all $x \in X$, the sequence
$$ 0 \rightarrow \mathscr{O}_{Y, \pi(x)} \rightarrow \mathscr{O}_{X,x} \rightarrow \mathscr{F}_x $$
is exact.
Pick an $i \in I$ such that $x \in U_i$.
Then $\mathscr{O}_{Y, \pi(x)} = \mathscr{O}_{V_i, \pi_i(x)}$, $\mathscr{O}_{U_i,x} = \mathscr{O}_{X,x}$ and $\mathscr{F}_x = \left( \iota_i^{\#} \mathscr{F} \right)_x$ by Lemma \ref{lem:etale_pullback}.
The claim follows. \par

It remains to check Property \ref{geometric:functor}.
By Lemma \ref{lem:functor_geometric_leaves}, it is enough to prove that, for all morphisms $\mathrm{Spec}\,L \rightarrow Y$ where $L$ is an algebraically closed field, $X \times_Y \mathrm{Spec}\,L$ is a smooth algebraic leaf of $\mathscr{F}$ over $\mathrm{Spec} L$.
Define $Z_i = U_i \times_{Y} \mathrm{Spec} L = U_i \times_{V_i} \mathrm{Spec} L$.
By assumption, whenver $Z_i$ is non-empty, it is a smooth algebraic leaf of $\iota_i^{\#} \mathscr{F}$ over $L$.
Note also that $\left\{ Z_i \right\}_{i \in I}$ is an open cover of $Z$.
It is shown that $Z \rightarrow X \times_S \mathrm{Spec}\,L$ is a  smooth algebraic leaf. \par

$U_i \times_S \mathrm{Spec}\,L$ is an affine open subset of $X \times_S \mathrm{Spec}\,L$.
Indeed the morphism $X \times_S \mathrm{Spec}\,L \rightarrow X$, being the base change of an affine morphism, is affine.
Furthermore $\left\{ U_i \times_S \mathrm{Spec}\,L \right\}_{i \in I}$ is an open cover of $X$.
By assumption, $Z_i \rightarrow U_i \times_S \mathrm{Spec}\,L$ is  a closed immersion, hence, by \cite[\href{https://stacks.math.columbia.edu/tag/01QO}{Lemma 01QO}]{stacks-project}, $Z \rightarrow X \times_{S} \mathrm{Spec}\,L$ is a closed immersion.
\ref{tangent:immersion} is satisfied. \par

Being smooth is a local property (\cite[\href{https://stacks.math.columbia.edu/tag/01V6}{Lemma 01V6}]{stacks-project}), hence \ref{tangent:smooth} is satisfied. \par

Being reduced is a local property, hence showing that $Z$ is integral amounts to showing that $Z$ is irreducible.
To this end, it is shown that for any non-empty $Z_k$ and for any open set $U \subseteq Z$, $U \cap Z_k \neq \emptyset$.
Because $\left\{ Z_i \right\}_{i \in I}$ is an open cover of $Z$, there exists a $j \in I$ such that $U \cap Z_j \neq \emptyset$.
If it can be shown that $Z_{jk} = \left(U_j \cap U_k \right) \times_Y \mathrm{Spec}\,L$ is non-empty, then $U \cap Z_k \neq \emptyset$.
Indeed, in this case, since $Z_k$ is integral, $Z_{jk}$ would be a dense open subset of $Z_j$, hence $Z_{jk} \cap U$ would be non-empty.
But then
$$ Z_k \cap U \supseteq Z_{jk} \cap U \neq \emptyset.$$
To show that $Z_{jk}$ is non-empty, note that it is a fibre of the morphism $U_j \cap U_k \rightarrow Y$, hence it is enough to show that $V_{jk} = V_j \cap V_k$ (Lemma \ref{lem:intersection_geometric_quotient}) contains the image of $\mathrm{Spec}\,L$ in $Y$.
Because by construction $Z_j$ and $Z_k$ are non-empty, it follows that the image of $\mathrm{Spec}\,L$  is contained in both $V_j$ and $V_k$.
This shows that $Z_{jk}$ is non-empty so that \ref{tangent:integral} holds. \par

By Lemma \ref{lem:directed_stalk_local}, checking that $f : (Z, \Omega_{Z/L}^1 ) \rightarrow (X, \mathscr{F} )$ is a weakly \'{e}tale morphism of directed spaces over $S$ can also be checked locally.
This can be checked on the open cover $\left\{ Z_i \right\}_{i \in I}$ using Lemma \ref{lem:etale_pullback}.
This shows \ref{tangent:distribution}.
\end{proof}

\section{Sheaf of First Integrals}
\label{sec:first_integrals}

This section studies first integrals.
\S \ref{subsec:first_integrals} defines the sheaf of first integrals and shows basic properties.
\S \ref{subsec:algebraic_integrability} uses the notion of first integrals to define algebraic integrability of a distribution.
\S \ref{subsec:finite_generation} proves that, in some cases, first integrals are finitely generated.
\S \ref{subsec:first_integrals_two} shows further properties of the sheaf of first integrals.
In this section, $d_{\mathscr{F}}$ denotes $d_{\mathscr{F},0}$ (Remark \ref{rem:uniqueness_exterior_derivative}).

\subsection{Basic Properties of First Integrals}
\label{subsec:first_integrals}

This subsection introduces first integrals, local first integrals and proves that local first integrals are a direct limit of first integrals (Lemma \ref{lem:localisation_well_defined}).

\begin{definition}[First integrals]
\label{def:first_integral}
Let $k$ be a field and let $(X, \mathscr{F})$ be a directed space over $k$.
The \emph{sheaf of first integrals} $\mathscr{O}_X^{\mathscr{F}}$ is the kernel subsheaf of $\mathscr{O}_X$ by the exterior derivative $d_{\mathscr{F}} = d_{\mathscr{F},0}$.
A local section of $\mathscr{O}_X^{\mathscr{F}}$ over an open set $U \subseteq X$ is called a \emph{first integral} of $\mathscr{F}$ over $U$.
\end{definition}

\begin{remark}[Subsheaf first integrals]
When $X$ is an integral locally Noetherian scheme over a field $k$ and $\mathscr{F}$ is torsion-free, the definition of first integral coincides with the more traditional one.
Under the correspondence of Remark \ref{rem:distribution_corr}, if $\mathscr{T}_{\mathscr{F}} \subseteq \mathscr{T}_X$ is a subsheaf of the tangent sheaf, then a function $f \in \mathscr{O}_X(U)$ is referred as a first integral for $\mathscr{T}_{\mathscr{F}}$ if
$$ \partial(f) = 0 \quad \forall \partial \in \mathscr{T}_{\mathscr{F}}(U).$$
Equivalently,
$$df(\partial) = 0 \quad \forall \partial \in \mathscr{T}_{\mathscr{F}}(U).$$
But this means precisely that $df$ is mapped to $0$ under the morphism
$$\Omega_X^1(U) \rightarrow \mathscr{F}(U) \rightarrow \mathscr{F}^{[1]}(U).$$
Now, since $\mathscr{F}(U)$ is torsion-free, $\mathscr{F}(U) \rightarrow \mathscr{F}^{[1]}(U)$ is injective (\cite[\href{https://stacks.math.columbia.edu/tag/0AY2}{Lemma 0AY2}]{stacks-project}), hence $df$ is mapped to $0 \in \mathscr{F}(U)$.
By definition, this happens if and only if $d_{\mathscr{F}}f = 0$.
\end{remark}

\begin{lemma}[First integrals are $k$-algebras]
\label{lem:sheaf_algebras}
Let $k$ be a field and let $(X, \mathscr{F})$ be a directed space over $k$.
Then $\mathscr{O}_X^{\mathscr{F}}$ is a sheaf of $k$-algebras. 
\end{lemma}

\begin{proof}
Since the kernel presheaf is already a sheaf,
$$\mathscr{O}_X^{\mathscr{F}}(U) = \left\{ f \in \mathscr{O}_X(U) \, | \, d_{\mathscr{F}} f = 0 \in \mathscr{F}(U) \right\}.$$
Certainly the constant functions $k \hookrightarrow \mathscr{O}_X(U)$ are first integrals.
Now, closure under subtraction and multiplication easily follows from the addition rule and the product rule respectively (Lemma \ref{lem:product_chain_rule}).
It follows that $\mathscr{O}_X^{\mathscr{F}}(U)$ is a $k$-algebra. 
\end{proof}

\begin{lemma}[Differentiation rules]
\label{lem:product_chain_rule}
Let $k$ be a field and let $(X, \mathscr{F})$ be a directed space over $k$.
Write $q : \Omega_{X/S}^1 \rightarrow \mathscr{F}$ for the quotient morphism.
Let $U \subseteq X$ be an open set of $X$, let $f,g \in \mathscr{O}_X(U)$ be local sections and let $p(t) \in \mathscr{O}_X^{\mathscr{F}}(U)[t]$ be a polynomial with coefficients in the ring of first integrals of $\mathscr{F}$ over $U$.
Then
\begin{itemize}
\item \emph{(Addition rule)}. $d_{\mathscr{F}}(f+g) = d_{\mathscr{F}}f + d_{\mathscr{F}}g$.
\item \emph{(Product rule)}. $d_{\mathscr{F}}(fg) = f \cdot d_{\mathscr{F}}g + g \cdot d_{\mathscr{F}}f$.
\item \emph{(Chain rule)}. $d_{\mathscr{F}}(p(f)) = p^{\prime}(f) \cdot d_{\mathscr{F}}f$.
\end{itemize}
\end{lemma}

\begin{proof}
The \emph{addition rule} and the \emph{product rule} are an easy consequence of the corresponding properties of $d$ together with the fact that $q$ is $\mathscr{O}_X(U)$-linear. \par

For the \emph{chain rule}, write $p(t) = \sum_{i=0}^d a_i t^i$ where $d$ is the degree of $p$ and $a_i \in \mathscr{O}_X^{\mathscr{F}}(U)$ for all $i$.
Firstly note that applying the product rule repeatedly yields
$$ d_{\mathscr{F}} f^i = i\,f^{i-1} \cdot d_{\mathscr{F}}f.$$
But then
\begin{align*}
d_{\mathscr{F}}(p(f)) &= d_{\mathscr{F}} \left( \sum_{i=0}^d a_i f^i \right) &\\
&= \sum_{i=0}^d \left( d_{\mathscr{F}}(a_i f^i) \right) 
& \text{(Addition Rule)} \\
&= \sum_{i=0}^d \left(a_i \cdot d_{\mathscr{F}} f^i + f^i \cdot d_{\mathscr{F}}a_i \right) & \text{(Product Rule)}\\
&= \sum_{i=0}^d a_i \cdot d_{\mathscr{F}} f^i & (d_{\mathscr{F}} a_i = 0) \\
&= \sum_{i=0}^d \left(i\, a_i f^{i-1} \right) \cdot d_{\mathscr{F}}f & \text{(Product Rule)} \\
&= p^{\prime}(f) \cdot d_{\mathscr{F}}f. &
\end{align*}
\end{proof}

\begin{lemma}[First integrals and localisation]
\label{lem:distinguished_first_integrals}
Let $k$ be a field and let $(X, \mathscr{F})$ be an affine directed space over $k$.
Suppose that $f \in \Gamma \left( X, \mathscr{O}_X^{\mathscr{F}} \right)$.
Then
$$ \Gamma \left( X_f , \mathscr{O}_X^{\mathscr{F}} \right) = \Gamma \left( X, \mathscr{O}_X^{\mathscr{F}} \right)_f.$$
\end{lemma}

\begin{proof}
By definition, there is an exact sequence
$$ 0 \rightarrow \Gamma \left( X, \mathscr{O}_X^{\mathscr{F}} \right) \rightarrow \Gamma \left( X, \mathscr{O}_X \right) \rightarrow \Gamma \left( X, \mathscr{F} \right).$$
By Lemma \ref{lem:product_chain_rule}, this is an exact sequence of $\Gamma \left( X, \mathscr{O}_X^{\mathscr{F}} \right)$-modules. Localising by $f$ preserves exactness, hence
$$ 0 \rightarrow \Gamma \left( X, \mathscr{O}_X^{\mathscr{F}} \right)_f \rightarrow \Gamma \left( X, \mathscr{O}_X \right)_f \rightarrow \Gamma \left( X, \mathscr{F} \right)_f.$$
is exact.
But then, since $X$ is affine, $\Gamma \left( X, \mathscr{O}_X \right)_f = \Gamma \left( X_f, \mathscr{O}_X \right)$ and $\Gamma \left( X, \mathscr{F} \right)_f = \Gamma \left( X_f, \mathscr{F} \right)$.
This shows that 
$$ \Gamma \left( X, \mathscr{O}_X^{\mathscr{F}} \right)_f =
\mathrm{ker}\left( \Gamma \left( X_f, \mathscr{O}_X \right) \rightarrow \Gamma \left( X_f, \mathscr{F} \right) \right) =  \Gamma \left( X_f , \mathscr{O}_X^{\mathscr{F}} \right).$$
\end{proof}

\begin{definition}[Local first integrals]
\label{def:local_first_integrals}
Let $k$ be a field and let $(X, \mathscr{F})$ be a directed space over $k$.
Write $q : \Omega_{X/S}^1 \rightarrow \mathscr{F}$ for the quotient morphism.
For any $x \in X$, the \emph{ring of local first integrals} at $x$ is
$$ \mathscr{O}_{X,x}^{\mathscr{F}} = \mathrm{ker} \left( d_{\mathscr{F},x} :  \mathscr{O}_{X,x} \xrightarrow{d_x} \Omega_{\mathscr{O}_{X,x}/k}^1 = \Omega_{X/k,x}^1 \xrightarrow{q_x} \mathscr{F}_x\right).$$
This is a $k$-algebra.
\end{definition}

\begin{remark}[Local first integrals are well-defined]
\label{rem:local_well_defined}
Some properties of the objects considered in Definition \ref{def:local_first_integrals} need to be checked.
Firstly, the equality $\Omega_{\mathscr{O}_{X,x}/k}^1 = \Omega_{X/k,x}^1$ follows from Lemma \ref{lem:cotangent_localisation}.
Secondly, $\mathscr{O}_{X,x}^{\mathscr{F}}$ is a $k$-algebra by applying Lemma \ref{lem:sheaf_algebras} to the local scheme $\mathrm{Spec}\,\mathscr{O}_{X,x}$ with the foliation $\mathscr{F}_x$.
\end{remark}

\begin{lemma}
\label{lem:cotangent_localisation}
Let $S$ be a scheme and let $X$ be a scheme over $S$. Let $x \in X$ be a point and let $s \in S$ be its image. Then
$$ \Omega_{X/S,x}^1 = \Omega_{\mathscr{O}_{X,x}/\mathscr{O}_{S,s}}^1.$$
Furthermore, the exterior derivatives
\begin{center}
\begin{tikzcd}
\mathscr{O}_X \arrow[r, "d"] \arrow[d] & \Omega_{X/S}^1 \arrow[d] \\
\mathscr{O}_{X,x} \arrow[r, "d_x"'] & \Omega_{X/S,x}^1
\end{tikzcd}
\end{center}
are compatible.
\end{lemma}

\begin{proof}
Combine \cite[\href{https://stacks.math.columbia.edu/tag/01US}{Lemma 01US}]{stacks-project} and \cite[\href{https://stacks.math.columbia.edu/tag/00RT}{Lemma 00RT}]{stacks-project}.
\end{proof}

\begin{lemma}[Local structure of local first integrals]
\label{lem:morphism_local_rings}
Let $k$ be a field and let $(X, \mathscr{F})$ be an integral directed space over $k$.
Then, the ring of local first integrals $\mathscr{O}_{X,x}^{\mathscr{F}}$ is a local ring and the morphism 
$$\mathscr{O}_{X,x}^{\mathscr{F}} \rightarrow \mathscr{O}_{X,x}$$
is a morphism of local rings.
\end{lemma}

\begin{proof}
Let $\mathfrak{m}_x$ be the maximal ideal of $\mathscr{O}_{X,x}$. 
It is shown that the inverse image of $\mathfrak{m}_x$ by $\mathscr{O}_{X,x}^{\mathscr{F}} \rightarrow \mathscr{O}_{X,x}$,
$$ \mathfrak{n}_x = \left\{ f \in \mathscr{O}_{X,x} \, | \, d_{\mathscr{F},x}(f) = 0 \text{ and } f \in \mathfrak{m}_x \right\}, $$
is the maximal ideal of $\mathscr{O}_{X,x}^{\mathscr{F}}$.
This simultaneously shows both statements.
$\mathfrak{n}_x$ clearly is an ideal.
To show that it is the maximal ideal, it is shown that any element of $\mathscr{O}_{X,x}^{\mathscr{F}}$ not contained in $\mathfrak{n}_x$ is a unit.
Let $u$ be such an element.
By assumption, $d_{\mathscr{F},x}(u) = 0$ and $u \notin \mathfrak{m}_x$.
As a result, $u$ is a unit in $\mathscr{O}_{X,x}$, hence there exists a $v \in \mathscr{O}_{X,x}$ such that $uv = 1$.
Applying the exterior derivative yields
$$ u \cdot d_{\mathscr{F},x}(v) = 0. $$
But $u$ is a unit in $\mathscr{O}_{X,x}$, hence $d_{\mathscr{F},x}(v) = 0$ and $v \in \mathscr{O}_{X,x}^{\mathscr{F}}$.
This shows that $u$ is a unit in $\mathscr{O}_{X,x}^{\mathscr{F}}$.
\end{proof}

\begin{lemma}[Local first integrals are compatible]
\label{lem:commutative_description}
Let $k$ be a field and let $(X, \mathscr{F})$ be an integral directed space over $k$.
Let $U \subseteq X$ be an open set of $X$ and let $x \in U$.
Then, there exists a commutative diagram
\begin{center}
\begin{tikzcd}
\mathscr{O}_X^{\mathscr{F}}(U) \arrow[r] \arrow[d] & \mathscr{O}_X(U) \arrow[d] \\
\mathscr{O}_{X,x}^{\mathscr{F}} \arrow[r] & \mathscr{O}_{X,x}.
\end{tikzcd}
\end{center}
\end{lemma}

\begin{proof}
By Lemma \ref{lem:cotangent_localisation}, there is a commutative diagram
\begin{center}
\begin{tikzcd}
\mathscr{O}_X^{\mathscr{F}}(U) \arrow[r] & \mathscr{O}_X(U) \arrow[d] \arrow[r, "d_U"] & \mathscr{F}(U) \arrow[d] \\
\mathscr{O}_{X,x}^{\mathscr{F}} \arrow[r] & \mathscr{O}_{X,x}  \arrow[r, "d_x"'] & \mathscr{F}_x.
\end{tikzcd}
\end{center}
By the universal property of the kernel
$$ \mathscr{O}_{X,x}^{\mathscr{F}} = \mathrm{ker} \left( \mathscr{O}_{X,x}  \rightarrow \mathscr{F}_x \right),$$
there exists a unique morphism of abelian groups $\mathscr{O}_X^{\mathscr{F}}(U) \rightarrow \mathscr{O}_{X,x}^{\mathscr{F}}$.
By Lemma \ref{lem:morphism_groups_rings}, this is a morphism of rings.
\end{proof}

\begin{lemma}[Local first integrals and Cartesian diagram]
\label{lem:cartesian_description}
Let $k$ be a field and let $(X, \mathscr{F})$ be an integral directed space over $k$.
Suppose that $\mathscr{F}$ is torsion-free.
Let $U \subseteq X$ be an open set of $X$ and let $x \in U$. Then, the diagram
\begin{center}
\begin{tikzcd}
\mathscr{O}_X^{\mathscr{F}}(U) \arrow[r] \arrow[d] & \mathscr{O}_X(U) \arrow[d] \\
\mathscr{O}_{X,x}^{\mathscr{F}} \arrow[r] & \mathscr{O}_{X,x}
\end{tikzcd}
\end{center}
is Cartesian in the category of rings.
\end{lemma}

\begin{proof}
Consider once again the diagram from Lemma \ref{lem:commutative_description}. 
\begin{center}
\begin{tikzcd}
\mathscr{O}_X^{\mathscr{F}}(U) \arrow[r] \arrow[d] & \mathscr{O}_X(U) \arrow[d] \arrow[r, "d_U"] & \mathscr{F}(U) \arrow[d] \\
\mathscr{O}_{X,x}^{\mathscr{F}} \arrow[r] & \mathscr{O}_{X,x}  \arrow[r, "d_x"'] & \mathscr{F}_x.
\end{tikzcd}
\end{center}
Let $C$ be a ring and suppose that there is a cone
\begin{center}
\begin{tikzcd}
C \arrow[r] \arrow[d] & \mathscr{O}_X(U) \arrow[d] \\
\mathscr{O}_{X,x}^{\mathscr{F}} \arrow[r] & \mathscr{O}_{X,x}.
\end{tikzcd}
\end{center}
It is shown that there exists a unique morphism $C \rightarrow \mathscr{O}_X^{\mathscr{F}}(U)$. 
Observe that the composition $C \rightarrow \mathscr{O}_{X,x}^{\mathscr{F}} \rightarrow \mathscr{O}_{X,x} \rightarrow \mathscr{F}_x$ is zero. On the other hand, because $\mathscr{O}_X(U)$ is a domain and $\mathscr{F}(U)$ is torsion-free, the morphism $\mathscr{F}(U) \rightarrow \mathscr{F}_x$ is injective (\cite[\href{https://stacks.math.columbia.edu/tag/0AXS}{Lemma 0AXS}]{stacks-project}). As a result the composition $C \rightarrow \mathscr{O}_X(U)\rightarrow \mathscr{F}(U)$ is zero. But then, by the universal property of the kernel $\mathscr{O}_X^{\mathscr{F}}(U)$, there exists a unique morphism of abelian groups $C \rightarrow \mathscr{O}_X^{\mathscr{F}}(U)$.
By Lemma \ref{lem:morphism_groups_rings}, this is a morphism of rings.
\end{proof}

\begin{lemma}[Localisation of first integrals]
\label{lem:localisation_well_defined}
Let $k$ be a field and let $(X, \mathscr{F})$ be an integral directed space over $k$.
Suppose that $\mathscr{F}$ is torsion-free.
Let $x \in X$, then
$$ \mathscr{O}_{X,x}^{\mathscr{F}} = \lim_{\substack{\longrightarrow \\ U \ni x}} \mathscr{O}_X^{\mathscr{F}}(U),$$
where the limit is taken over all affine open subsets containing $x$.
\end{lemma}

\begin{proof}
Firstly, recall that, in the category of rings, filtered colimits commute with finite limits (\cite[\href{https://stacks.math.columbia.edu/tag/002W}{Lemma 002W}]{stacks-project}).
Now, let $U$ be an affine open subset of $X$ containing $x$. By Lemma \ref{lem:cartesian_description}, it follows that
\begin{align}
\label{eq:limit_open_set}
\mathscr{O}_X^{\mathscr{F}}(U) = \mathscr{O}_X(U) \times_{\mathscr{O}_{X,x}}
\mathscr{O}_{X,x}^{\mathscr{F}}.
\end{align}
Now, taking the limit over all affine open subsets yields
\begin{align*}
\lim_{\substack{\longrightarrow \\ U \ni x}} \mathscr{O}_X^{\mathscr{F}}(U)
&=
\lim_{\substack{\longrightarrow \\ U \ni x}} \left( \mathscr{O}_X(U) \times_{\mathscr{O}_{X,x}}
\mathscr{O}_{X,x}^{\mathscr{F}} \right)
& \text{(\ref{eq:limit_open_set})} \\
&=
\left( \lim_{\substack{\longrightarrow \\ U \ni x}} \mathscr{O}_X(U) \right) \times_{\mathscr{O}_{X,x}}
\mathscr{O}_{X,x}^{\mathscr{F}}
& \substack{\text{commutativity} \\ \text{of limits and colimits}} \\
&=
\mathscr{O}_{X,x} \times_{\mathscr{O}_{X,x}}
\mathscr{O}_{X,x}^{\mathscr{F}} & \\
&= \mathscr{O}_{X,x}^{\mathscr{F}}. & 
\end{align*}
\end{proof}

\subsection{Algebraic Integrability}
\label{subsec:algebraic_integrability}

This subsection defines algebraic integrability of a distribution in terms of the dimension of its first integrals and proves a well-known characterisation (Lemma \ref{lem:traditional_algebraic_integrability}).
The material in this section is inspired by the main results of \cite{MR1017286}.

\begin{definition}[Rank and corank]
\label{def:rank_corank}
Let $k$ be a field and let $(X, \mathscr{F})$ be an integral directed space locally of finite type over $k$.
Let $\eta$ be the generic point of $X$.
By Remark \ref{rem:coherence_distributions}, $\Omega_{X/k}^1$ and $\mathscr{F}$ are coherent $\mathscr{O}_X$-modules. The \emph{rank of $\mathscr{F}$} is defined to be its rank as a coherent $\mathscr{O}_X$-module. The \emph{corank of $\mathscr{F}$} is defined to be the difference between the rank of $\Omega_{X/k}^1$ and the rank of $\mathscr{F}$.
\end{definition}

\begin{lemma}[Dimension of first integrals]
\label{lem:dim_first_integrals}
Let $k$ be a field of characteristic zero and let $(X, \mathscr{F})$ be an integral directed space locally of finite type over $k$.
Let $\eta$ be the generic point of $X$, then
$$ \mathrm{rank}\,\mathscr{F} \leq \mathrm{tr\,deg}_{\mathscr{O}_{X,\eta}^{\mathscr{F}}} \mathscr{O}_{X,\eta}.$$
\end{lemma}

\begin{proof}
It is firstly shown that there is a surjection of $\mathscr{O}_{X,\eta}$-vector spaces
\begin{align}
\label{eq:dimension_surjection}
\Omega_{\mathscr{O}_{X,\eta}/\mathscr{O}_{X,\eta}^{\mathscr{F}}}^1 \rightarrow \mathscr{F} \otimes_{\mathscr{O}_X} \kappa(\eta) \rightarrow 0.
\end{align}
The relative cotangent exact sequence for $k \rightarrow \mathscr{O}_{X,\eta}^{\mathscr{F}} \rightarrow \mathscr{O}_{X,\eta}$ gives that 
$$ \Omega_{\mathscr{O}_{X,\eta}/\mathscr{O}_{X,\eta}^{\mathscr{F}}}^1 = \mathrm{coker} \left( \Omega_{\mathscr{O}_{X,\eta}^{\mathscr{F}}/k}^1 \otimes_{\mathscr{O}_X} \kappa(\eta) \rightarrow \Omega_{\mathscr{O}_{X,\eta}/k}^1 \right).$$
Hence, it is enough to show that the composition 
$$ \Omega_{\mathscr{O}_{X,\eta}^{\mathscr{F}}/k}^1 \otimes_{\mathscr{O}_X} \kappa(\eta) \rightarrow \Omega_{\mathscr{O}_{X,\eta}/k}^1 = \Omega_{X/k}^1 \otimes_{\mathscr{O}_X} \kappa(\eta) \rightarrow \mathscr{F} \otimes_{\mathscr{O}_X} \kappa(\eta)$$
is zero. This follows because, by construction, the $k$-derivation
$$ \mathscr{O}_{X,\eta}^{\mathscr{F}} \rightarrow \mathscr{F} \otimes_{\mathscr{O}_X} \kappa(\eta)$$
is zero, hence by the universal property of the sheaf of differentials, the morphism 
$$ \Omega_{\mathscr{O}_{X,\eta}^{\mathscr{F}}/k}^1 \otimes_{\mathscr{O}_X} \kappa(\eta) \rightarrow 
\mathscr{F} \otimes_{\mathscr{O}_X} \kappa(\eta) $$
is zero. \par

Now, by (\ref{eq:dimension_surjection}), it is clear that
$$ \mathrm{dim}_{\kappa(\eta)}\, \mathscr{F} \otimes_{\mathscr{O}_X} \kappa(\eta) = \mathrm{rank}\, \mathscr{F} \leq \mathrm{dim}_{\kappa(\eta)}\, \Omega_{\mathscr{O}_{X,\eta}/\mathscr{O}_{X,\eta}^{\mathscr{F}}}^1.$$
But since $k$ is a field of characteristic zero and $\kappa(\eta)$ is finitely generated over $k$, $\mathscr{O}_{X,\eta}^{\mathscr{F}} \rightarrow \mathscr{O}_{X,\eta}$ is a separably generated field extension (Lemma \ref{lem:morphism_local_rings}).
Hence
$$ \mathrm{tr\,deg}_{\mathscr{O}_{X,\eta}^{\mathscr{F}}} \mathscr{O}_{X,\eta} = \mathrm{dim}_{\kappa(\eta)}\, \Omega_{\mathscr{O}_{X,\eta}/\mathscr{O}_{X,\eta}^{\mathscr{F}}}^1.$$
This proves the claim.
\end{proof}

\begin{definition}[Algebraically integrable distributions]
\label{def:algebraically_integrable}
Let $k$ be a field of characteristic zero and let $(X, \mathscr{F})$ be an integral directed space locally of finite type over $k$.
\emph{$\mathscr{F}$ is algebraically integrable} if 
$$ \mathrm{rank}\,\mathscr{F} = \mathrm{tr\,deg}_{\mathscr{O}_{X,\eta}^{\mathscr{F}}} \mathscr{O}_{X,\eta}.$$
\end{definition}

\begin{lemma}[Characterisation of algebraic integrability]
\label{lem:traditional_algebraic_integrability}
Let $k$ be a field of characteristic zero and let $(X, \mathscr{F})$ be an integral directed space locally of finite type over $k$.
Let $\eta$ be the generic point of $X$.
$\mathscr{F}$ is algebraically integrable if and only if there exists an affine open set $\iota : U \rightarrow X$ such that the induced $k$-morphism
$$ \pi : U = \mathrm{Spec}\, \Gamma \left( U, \mathscr{O}_X \right) \rightarrow \mathrm{Spec}\, \Gamma \left( U, \mathscr{O}_X^{\mathscr{F}} \right) = V $$
gives an isomorphism of locally free sheaves
$$ \Omega_{U/V}^1 \xrightarrow{\sim} \iota^* \mathscr{F}.$$
\end{lemma}

\begin{proof}
\textbf{($\rightarrow$)}.
Firstly, it is shown that there exists an open set $U \subseteq X$ such that
$$ \Gamma \left( U, \mathscr{O}_X^{\mathscr{F}} \right) \rightarrow \mathscr{O}_{X, \eta} $$
is a localisation.
In other words, $\mathscr{O}_{X,\eta}^{\mathscr{F}}$ is the field of fractions of $\mathscr{O}_X^{\mathscr{F}}(U)$.
By Lemma \ref{lem:morphism_local_rings}, $\mathscr{O}_{X,\eta}^{\mathscr{F}}$ is a subfield of $\mathscr{O}_{X,\eta}$. Since $X$ is locally of finite type over $k$, $\mathscr{O}_{X,\eta}$ is finitely generated as a field over $k$.
This implies that $\mathscr{O}_{X,\eta}^{\mathscr{F}}$ is also finitely generated as field over $k$. 
Let $\{f_i\}_{i \in I}$ be a finite set of generators. 
By Lemma \ref{lem:localisation_well_defined}, there exists an affine open subset $U \subseteq X$ such that $f_i \in \mathscr{O}_X^{\mathscr{F}}(U)$ for all $i \in I$. 
But then the field of fractions of $\mathscr{O}_X^{\mathscr{F}}(U)$ is $\mathscr{O}_{X,\eta}^{\mathscr{F}}$. 
Up to replacing $X$ by $U$ and noting that the restriction of an algebraically integrable distribution to a non-empty open subset remains algebraically integrable, it may be assumed that $X$ is affine and $\mathscr{O}_{X,\eta}^{\mathscr{F}}$ is the field of fractions of $\Gamma \left( X, \mathscr{O}_X^{\mathscr{F}} \right)$.
Let $Y = \mathrm{Spec} \, \Gamma \left( X, \mathscr{O}_X^{\mathscr{F}} \right)$. 
Using Lemma \ref{lem:invariant_complex}, it follows that the induced morphism $\pi^{\prime} : X \rightarrow Y$ is $\mathscr{F}$-invariant, hence there is a surjection
\begin{align}
\label{eq:almost_isomorphic}
\Omega_{X/Y}^1 \rightarrow \mathscr{F}.
\end{align}
Since $\mathscr{F}$ is algebraically integrable and $\mathscr{O}_{X,\eta}^{\mathscr{F}}$ is the field of fractions of $\Gamma \left( X, \mathscr{O}_X^{\mathscr{F}} \right)$,
\begin{align*}
\mathrm{rank}\, \Omega_{X/Y}^1
&= \mathrm{dim}_{\kappa(\eta)} \Omega_{X/Y}^1 \otimes_{\mathscr{O}_X} \kappa(\eta) & \\
&= \mathrm{dim}_{\kappa(\eta)} \Omega_{\mathscr{O}_{X, \eta}/\mathscr{O}_{X, \eta}^{\mathscr{F}}}^1
& \text{(Lemma \ref{lem:cotangent_localisation})} \\
&= \mathrm{tr\,deg}_{\mathscr{O}_{X, \eta}^{\mathscr{F}}} \mathscr{O}_{X,\eta} & \text{($\mathscr{O}_{X, \eta}^{\mathscr{F}} \subseteq \mathscr{O}_{X,\eta}$ separably generated)} \\
&= \mathrm{rank}\, \mathscr{F} & \text{($\mathscr{F}$ algebraically integrable).} 
\end{align*}

Let $\mathscr{K}$ be the kernel of the surjection in (\ref{eq:almost_isomorphic}). Since $\Omega_{X/Y}^1$ is coherent (Remark \ref{rem:coherence_distributions}), $\mathscr{K}$ is coherent and its rank is zero.
Now, by generic flatness (\cite[\href{https://stacks.math.columbia.edu/tag/052B}{Proposition 052B}]{stacks-project}) applied to the identity $X \rightarrow X$, there exists a dense open set $U \subseteq X$ such that $\Omega_{X/Y}^1|_U$ is flat over $U$.
Since $\Omega_{X/Y}^1$ is coherent, it is locally free over $U$ (\cite[\href{https://stacks.math.columbia.edu/tag/05P2}{Lemma 05P2}]{stacks-project}).
But then $\mathscr{K}|_U$ is a rank zero subsheaf of a locally free sheaf, hence it has to be zero and 
$$\Omega_{U/Y}^1 = \Omega_{X/Y}^1|_U \rightarrow \iota^{\#} \mathscr{F}$$
is an isomorphism.
Lastly, up to replacing $U$ by an affine subset, assume $U$ is affine. Let $V = \mathrm{Spec} \, \Gamma \left( U, \mathscr{O}_X^{\mathscr{F}} \right)$ and let $\pi : U \rightarrow V$ be the induced morphism.
Using the cotangent exact sequence for $U \rightarrow V \rightarrow Y$ yields a surjection
\begin{align}
\label{eq:cotangent_restriction}
\iota^{\#} \mathscr{F} = \Omega_{U/Y}^1 \rightarrow \Omega_{U/V}^1.
\end{align}
But again by Lemma \ref{lem:invariant_complex}, $\pi$ is $\iota^{\#} \mathscr{F}$-invariant and there is a surjection
$$ \Omega_{U/V}^1 \rightarrow \iota^{\#} \mathscr{F} $$
which, in light of (\ref{eq:cotangent_restriction}), must be an isomorphism of locally free sheaves. \par

\textbf{($\leftarrow$)}.
For ease of notation, let $M$ denote the field of fractions of $\Gamma \left(U, \mathscr{O}_X^{\mathscr{F}} \right)$.
To prove that $\mathscr{F}$ is algebraically integrable, it is enough to show that 
$$ \mathrm{tr\, deg}_{M} \mathscr{O}_{X,\eta} = \mathrm{rank}\,\mathscr{F}.$$
Indeed, since $M \subseteq \mathscr{O}_{X,\eta}^{\mathscr{F}}$ is a field extension,
$$ \mathrm{tr\, deg}_{M} \mathscr{O}_{X,\eta} \geq \mathrm{tr\, deg}_{\mathscr{O}_{X,\eta}^{\mathscr{F}}} \mathscr{O}_{X,\eta}.$$
Now, localising the isomorphism $\Omega_{U/V}^1 \rightarrow \iota^* \mathscr{F}$ at the generic point $\eta$ gives an isomorphism
$$ \Omega_{U/V}^1 \otimes_{\mathscr{O}_U} \kappa(\eta) = \Omega_{\mathscr{O}_{X,\eta}/M} \xrightarrow{\sim} \mathscr{F} \otimes_{\mathscr{O}_U} \kappa(\eta) = \mathscr{F} \otimes_{\mathscr{O}_X} \kappa(\eta).$$
Finally since the field extension $M \subseteq \mathscr{O}_{X,\eta}$ is separably generated,
$$ \mathrm{dim}_{\kappa(\eta)} \Omega_{U/V}^1 \otimes_{\mathscr{O}_U} \kappa(\eta) = \mathrm{tr\,deg}_{M} \mathscr{O}_{X,\eta}.$$
This proves the result.
\end{proof}

\begin{lemma}[Algebraically integrable distributions are foliations]
\label{lem:alg_integrable_distribution}
Let $k$ be a field of characteristic zero and let $(X, \mathscr{F})$ be an integral directed space locally of finite type over $k$.
Suppose that $\mathscr{F}$ is algebraically integrable.
Then $\mathscr{F}$ is a foliation.
\end{lemma}

\begin{proof}
Pick an open set $\iota : U \rightarrow X$ and a morphism $\pi : U \rightarrow V$ satisfying the conclusions of Lemma \ref{lem:traditional_algebraic_integrability}. Note that $\Omega_{U/V}^1$ has an exterior derivative. This shows that $\mathscr{F}$ is involutive.
\end{proof}

\subsection{Finite Generation}
\label{subsec:finite_generation}

This subsection uses a theorem of Nagata (\cite{MR88034}) to show that the ring of first integrals is finitely generated when the corank of the distribution is less than or equal to two (Corollary \ref{cor:fin_gen}).
Remark \ref{rem:fin_gen} discusses open problems in this area.

\begin{definition}[Finitely generated first integrals]
\label{def:finite_first_integrals}
Let $k$ be a field and let $(X, \mathscr{F})$ be an integral directed space locally of finite type over $k$. \emph{The first integrals of $\mathscr{F}$ are locally of finite type over $k$} if for all affine open sets $U \subseteq X$, $\mathscr{O}_X^{\mathscr{F}}(U)$ is an algebra of finite type over $k$.
\end{definition}

\begin{corollary}[First integrals of low corank distribution are finitely generated]
\label{cor:fin_gen}
Let $k$ be a field of characteristic zero and let $(X, \mathscr{F})$ be a normal integral directed space locally of finite type over $k$.
Suppose that $\mathscr{F}$ is torsion-free and $\mathrm{corank}\,\mathscr{F} \leq 2$.
Then the first integrals of $\mathscr{F}$ are locally of finite type over $k$.
\end{corollary}

The result is an application of the following theorem of Nagata.

\begin{theorem}[Nagata's theorem]
\label{thm:fin_gen}
Let $k$ be a field and let $B$ be a normal integral finitely generated $k$-algebra.
Let $R(B)$ be the field of rational functions of $B$ and let $M$ be a subfield of $R(B)$ containing $k$. 
Suppose that
$$ \mathrm{tr\,deg}_k M \leq 2.$$
Then $A = B \cap M$ is a finitely generated $k$-algebra.
\end{theorem}

\begin{proof}
See the main theorem of \cite{MR88034}.
\end{proof}

\begin{remark}[Translation note]
\label{rem:translation_note}
The terminology used in \cite{MR88034} is based on the previous work \cite{MR82725} by the same author.
In particular, note that an \emph{affine ring} is a ring finitely generated over either a field or a Dedekind domain (\cite[\S 2]{MR82725}). 
\end{remark}

\begin{proof}[Proof of Corollary \ref{cor:fin_gen}]
Let $U \subseteq X$ be an affine non-empty open subset of $X$.
Write $A = \mathscr{O}_X^{\mathscr{F}}(U)$, $B = \mathscr{O}_X(U)$ and $M = \mathscr{O}_{X,\eta}^{\mathscr{F}}$.
By assumption, $B$ is a normal integral finitely generated $k$-algebra. 
Applying Lemma \ref{lem:cartesian_description} with $x = \eta \in U$ the generic point of $X$ yields
$$ A = B \cap M \subseteq R(B). $$
By Lemma \ref{lem:morphism_local_rings}, $M$ is a subfield of $R(B)$. 
By Lemma \ref{lem:dim_first_integrals},
$$ \mathrm{tr\,deg}_{k} M \leq \mathrm{corank} \,\mathscr{F} = 2.$$
Applying the Theorem gives that $A$ is a finitely generated $k$-algebra.
\end{proof}

\begin{remark}[Finite generation and log canonical singularities]
\label{rem:fin_gen}
In general, the first integrals of a torsion-free foliation need not be finitely generated.
The survey \cite[\S 3.1]{MR1917641} gives an example of a rank one foliation on $\mathbb{A}_k^7$ whose first integrals are not locally of finite type over $k$.
Further examples are given in \S 3.2.
The foliations given as counterexamples are all induced by a nilpotent vector field.
According to \cite[I.ii.4]{MR3128985}, a rank one foliation has this property if and only if it does not have log-canonical singularities.
It seems plausible to conjecture that the first integrals of a log-canonical algebraically integrable torsion-free foliation are locally of finite type over $k$.
\end{remark}

\subsection{Further Properties of First Integrals}
\label{subsec:first_integrals_two}

This subsection shows that the restriction maps of affine open sets in the sheaf of first integrals are birational (Proposition \ref{prop:birational_equivalence}).
This is a first step in showing that the restriction maps are open immersions.

\begin{proposition}[First integrals are algebraically closed]
\label{prop:alg_closed}
Let $k$ be a field of characteristic zero and let $(X, \mathscr{F})$ be an integral directed space over $k$. 
Assume that $\mathscr{F}$ is torsion-free and let $U \subseteq X$ be an open subset of $X$.
Then, the morphism of rings
$$ \mathscr{O}_X^{\mathscr{F}}(U) \rightarrow \mathscr{O}_X(U) $$
is algebraically closed.
\end{proposition}

\begin{proof}
Let $p(t) \in \mathscr{O}_X^{\mathscr{F}}(U)[t]$ be a non-zero polynomial with coefficients in the ring of first integrals over $U$.
Suppose that there exists an $f \in \mathscr{O}_X(U)$ such that $p(f) = 0$.
By definition of algebraically closed ring extension, the Lemma follows if it can be shown that $f \in \mathscr{O}_X^{\mathscr{F}}(U)$.
This is achieved by induction on $d = \mathrm{deg}\,p$. \par

\textbf{Base case.} The case $d = 0$ is trivial since any non-zero polynomial of degree zero cannot have a root. \par

\textbf{Inductive Step.} By the chain rule of Lemma \ref{lem:product_chain_rule},
$$ 0 = d_{\mathscr{F}}(p(f)) = p^{\prime}(f) \cdot d_{\mathscr{F}}f. $$
Now, $\mathscr{F}(U)$ is a torsion-free module over the domain $\mathscr{O}_X(U)$.
It follows that either $d_{\mathscr{F}}f = 0$, implying that $f \in \mathscr{O}_X^{\mathscr{F}}(U)$, or $p^{\prime}(f) = 0$.
But $p^{\prime}(t)$ is a polynomial with coefficients in $\mathscr{O}_X^{\mathscr{F}}(U)$.
Furthermore, because $X$ is defined over a field of characteristic zero, the initial morphism $\mathbb{Z} \rightarrow \mathscr{O}_X(U)$ is injective.
This implies that the degree $d-1$ coefficient of $p^{\prime}(t)$ is non-zero.
The inductive hypothesis is now satisfied.
\end{proof}

\begin{lemma}[Descent of normality]
\label{lem:normal_domains}
Let $k$ be a field of characteristic zero and let $(X, \mathscr{F})$ be a normal integral directed space over $k$.
Assume that $\mathscr{F}$ is torsion-free and let $U \subseteq X$ be an open subset of $X$.
Then, $\mathscr{O}_X^{\mathscr{F}}(U)$ is a normal domain.
\end{lemma}

\begin{proof}
For ease of notation, set $A = \mathscr{O}_X^{\mathscr{F}}(U)$ and $B = \mathscr{O}_X(U)$.
Let $R(B)$ (\emph{respectively} $R(A)$) denote the field of fractions of the domain $B$ (\emph{respectively} $A$).
By Proposition \ref{prop:alg_closed}, $A \rightarrow B$ is an algebraically closed morphism of integral domains. 
In particular $A$ is integrally closed in $B$. Since $X$ is normal, $B$ is integrally closed in $R(B)$. By \cite[\href{https://stacks.math.columbia.edu/tag/0308}{Lemma 0308}]{stacks-project}, the integral closure of $A \rightarrow B \rightarrow R(B)$ is $A$.
But the morphism $A \rightarrow R(B)$ factors into $A \rightarrow R(A) \rightarrow R(B)$, hence $A$ is integrally closed in $R(A)$.
\end{proof}

\begin{lemma}[Restriction of first integrals is integrally closed]
\label{lem:integrally_closed}
Let $k$ be a field of characteristic zero and let $(X, \mathscr{F})$ be a normal integral directed space over $k$. Suppose that $\mathscr{F}$ is torsion-free.
Let $V \subseteq U \subseteq X$ be two affine open subsets of $X$.
Then, the morphism of rings
$$ \mathscr{O}_X^{\mathscr{F}}(U) \rightarrow \mathscr{O}_X^{\mathscr{F}}(V)$$
is integrally closed.
\end{lemma}

\begin{proof}
For ease of notation, set $A =  \mathscr{O}_X^{\mathscr{F}}(U)$, $A^{\prime} =  \mathscr{O}_X^{\mathscr{F}}(V)$ and $B =  \mathscr{O}_X(U)$
By Proposition \ref{prop:alg_closed}, $A$ is algebraically closed in $B$.
Since $X$ is normal, $B$ is integrally closed in $R(B)$. 
Therefore $A$ is integrally closed in $R(B)$.
Since $A^{\prime}$ is a subring of $R(B)$, $A$ is integrally closed in $A^{\prime}$.
\end{proof}

\begin{proposition}[Restriction of first integrals is birational]
\label{prop:birational_equivalence}
Let $k$ be a field of characteristic zero and let $(X, \mathscr{F})$ be a normal integral directed space over $k$. Suppose that $\mathscr{F}$ is torsion-free and the first integrals of $\mathscr{F}$ are locally of finite type over $k$.
Let $V \subseteq U \subseteq X$ be two affine open subsets of $X$.
Suppose that the injective restriction morphism
$$\rho : \mathscr{O}_X^{\mathscr{F}}(U) \rightarrow \mathscr{O}_X^{\mathscr{F}}(V)$$
induces an algebraic extension of fields
$$ R\left(\mathscr{O}_X^{\mathscr{F}}(U)\right) \rightarrow R\left(\mathscr{O}_X^{\mathscr{F}}(V)\right).$$
Then $\rho$ is a birational morphism and
$$R\left(\mathscr{O}_X^{\mathscr{F}}(U)\right) \cong R\left(\mathscr{O}_X^{\mathscr{F}}(V)\right).$$
\end{proposition}

\begin{proof}
For ease of notation, set $A = \mathscr{O}_X^{\mathscr{F}}(U)$ and $A^{\prime} = \mathscr{O}_X^{\mathscr{F}}(V)$.
The restriction morphism $\rho : A \rightarrow A^{\prime}$ is injective since $X$ is integral.
It is known that $R(A) \rightarrow R(A^{\prime})$ is an algebraic extension.
It is shown that this is an isomorphism of fields. \par

By assumption $\mathrm{tr\,deg}_{R(A)} R(A^{\prime}) = 0$ and $A$ and $A^{\prime}$ are finitely generated over $k$.
In particular $A^{\prime}$ is finitely generated over $A$. This implies that
$$\mathrm{Spec}\,\rho : \mathrm{Spec}\,A^{\prime} \rightarrow \mathrm{Spec}\,A$$
is of finite type and separated.
By the implication $(1) \rightarrow (5)$ of \cite[\href{https://stacks.math.columbia.edu/tag/02NX}{Lemma 02NX}]{stacks-project}, there exists an open set $V \subseteq \mathrm{Spec}\,A$ such that the morphism $\rho\,|_{\rho^{-1}(V)} : \rho^{-1}(V) \rightarrow V$ is finite.
Shrinking $V$ if necessary, it may be assumed that there exists a $g \in A$ such that $V = D(g) = \mathrm{Spec}\,A_g$ is a distinguished open set in $\mathrm{Spec}\, A$.
As a consequence $\rho^{-1}(V) = \mathrm{Spec}\, A_g^{\prime}$.
It follows that the morphism $A_g \rightarrow A_g^{\prime}$ is finite, hence integral.
On the other hand, the morphism $A \rightarrow A^{\prime}$ is integrally closed (Lemma \ref{lem:integrally_closed}).
Because localisation and integral closure commute (\cite[\href{https://stacks.math.columbia.edu/tag/0307}{Lemma 0307}]{stacks-project}), the morphism $A_g \rightarrow A_g^{\prime}$ is integrally closed. But now $A_g \rightarrow A_g^{\prime}$ is both integral and integrally closed.
This implies that $A_g \cong A_g^{\prime}$, so that $\rho$ is a birational morphism.
Finally, because both $A$ and $A^{\prime}$ are of finite type over $k$, applying \cite[\href{https://stacks.math.columbia.edu/tag/0552}{Lemma 0552}]{stacks-project} yields $R(A) \cong R(A^{\prime})$.
\end{proof}

\section{Stability}
\label{sec:stability}

This section introduces a stability condition.
\S \ref{subsec:stability_quotient} defines stability and shows that a neighborhood of a stable point has a geometric quotient.
\S \ref{subsec:stability_generic} shows that the generic point of an integral scheme is stable.

\subsection{Stability and Geometric Quotients}
\label{subsec:stability_quotient}

This subsection defines stability and construct a local geometric quotient around a stable point.
To this end, it is necessary to show that the restriction maps of affine open sets in the sheaf of first integrals are open immersions.
This is achieved using the results of \S \ref{subsec:first_integrals_two} together with the flatness assumption.

\begin{definition}[$\mathscr{F}$-stability]
\label{def:stability}
Let $S$ be a scheme and let $(X, \mathscr{F})$ be an integral directed space over $S$.
Let $\eta$ be generic point of $X$.
A point $x \in X$ \emph{is $\mathscr{F}$-stable} if there exists an affine open subset $U \subseteq X$ containing $x$ whose induced morphism
$$ A = \Gamma \left( U, \mathscr{O}_X^{\mathscr{F}} \right) \rightarrow \Gamma \left( U, \mathscr{O}_X \right) = B$$
has the following properties.
\begin{enumerate}[label={(\alph*)}]
\item \label{stable:smooth} It is smooth.
\item \label{stable:dimension} It is of relative dimension $\mathrm{rank}\, \mathscr{F}$.
\item \label{stable:connected} It has geometrically connected fibres.
\end{enumerate}
\end{definition}

\begin{remark}[Stability is an open condition]
\label{rem:stable_open}
It follows from the definition that being $\mathscr{F}$-stable is an open condition.
\end{remark}

\begin{lemma}[Stability and smoothness]
\label{lem:stable_smooth}
Let $S$ be a scheme and let $(X, \mathscr{F})$ be an integral directed space over $S$.
Suppose that $x \in X$ is an $\mathscr{F}$-stable point, then $\mathscr{F}_x$ is a free $\mathscr{O}_{X,x}$-module.
\end{lemma}

\begin{proof}
Using the notation of Definition \ref{def:stability}, consider the induced surjection of modules
\begin{align}
\label{eq:surjection}
\Omega_{B/A}^1 \rightarrow F,
\end{align}
where $F = \Gamma(U, \mathscr{F})$.
Since $\Omega_{B/A}^1$ is locally free (Property \ref{stable:smooth}), it is enough to show that (\ref{eq:surjection}) is an isomorphism.
To this end, note that the modules $\Omega_{B/A}^1$ and $F$ have the same rank (Property \ref{stable:dimension}), hence the kernel submodule $K$ of (\ref{eq:surjection}) must be a torsion submodule of $\Omega_{B/A}^1$.
Since $.\Omega_{B/A}^1$ is torsion-free, $K = 0$.
\end{proof}

\begin{proposition}[Neighborhood of stable point is a geometric quotient]
\label{prop:stable_geometric quotient}
Let $k$ be a field of characteristic zero and let $(X, \mathscr{F})$ be a normal integral foliated space locally of finite type over $k$.
Suppose that $\mathscr{F}$ is algebraically integrable and its first integrals are locally of finite type over $k$. 
Suppose that $x \in X$ is $\mathscr{F}$-stable and let $\iota : U \subseteq X$ be the affine open set satisfying the hypothesis of Definition \ref{def:stability}.
Let $V = \mathrm{Spec}\,\Gamma \left( U, \mathscr{O}_X^{\mathscr{F}}\right)$. 
Then $\pi : U \rightarrow V$ is a geometric quotient.
\end{proposition}

\begin{proof}
Property \ref{geometric:invariant} is shown. Because $k$ is a field of characteristic zero, the morphism $\pi : U \rightarrow V$ is universally separable (Remark \ref{rem:char_zero_separable}).
Note that by construction the sequence
$$ 0 \rightarrow \mathscr{O}_V \rightarrow \pi_* \mathscr{O}_U \rightarrow \pi_* \mathscr{F}$$
is a complex.
Hence, by Lemma \ref{lem:invariant_complex}.(1), $\pi$ is $\mathscr{F}$-invariant. \par

Property \ref{geometric:functor} is shown.
By Lemma \ref{lem:functor_geometric_leaves}, it is enough to show that, if $\mathrm{Spec}\,L \rightarrow V$ is a morphism where $L$ is an algebraically closed field, then the fibre $Z = U \times_V \mathrm{Spec}\,L$ is a smooth algebraic leaf. 
By \cite[\href{https://stacks.math.columbia.edu/tag/01KR}{Lemma 01KR}]{stacks-project}, since $V \rightarrow \mathrm{Spec}\,k$ is affine (hence separated), the morphism $Z = U \times_V L \rightarrow X \times_k L$ is a closed immersion.
Thus \ref{tangent:immersion} holds.
Being smooth follows by base change.
Thus \ref{tangent:smooth} holds.
By assumption $Z \rightarrow \mathrm{Spec} \, L$ is a smooth connected locally Noetherian scheme over a field $L$.
Therefore \cite[\href{https://stacks.math.columbia.edu/tag/038X}{Lemma 038X}]{stacks-project} shows that $Z$ is, in particular, a normal scheme.
Now, using the $(1) \rightarrow (2)$ implication of \cite[\href{https://stacks.math.columbia.edu/tag/033N}{Lemma 033N}]{stacks-project} gives that $Z$ is a disjoint union of integral schemes.
Since $Z$ is irreducible, it is integral.
Thus \ref{tangent:integral} holds.
Note that $f : Z \rightarrow U$ is universally separable since it is a $k$-morphism where $k$ is a field of characteristic zero.
Furthermore, $\Omega_{Z/L}^1 = f^* \Omega_{U/V}^1 = f^* \left( \iota^{\#} \mathscr{F} \right)$.
This shows \ref{tangent:distribution}. \par

Property \ref{geometric:open} is shown.
Since $\pi$ is flat and locally of finite presentation, it is universally open (\cite[\href{https://stacks.math.columbia.edu/tag/01UA}{Lemma 01UA}]{stacks-project}). \par

Property \ref{geometric:exact} is shown.
It suffices to show that, for all $x \in X$, the sequence
$$ 0 \rightarrow \mathscr{O}_{Y, \pi(x)} \rightarrow \mathscr{O}_{X,x} \rightarrow \mathscr{F}_x$$
is exact.
In other words, it has to be shown that the induced morphism
$$ \mathscr{O}_{Y, \pi(x)} \rightarrow \mathscr{O}_{X,x}^{\mathscr{F}} $$
is an isomorphism.
This is proven in the next Lemma \ref{lem:local_exactness}.
\end{proof}

\begin{lemma}[Local exactness]
\label{lem:local_exactness}
Let $k$ be a field of characteristic zero and let $(X, \mathscr{F})$ be a normal integral foliated space locally of finite type over $k$.
Suppose that $\mathscr{F}$ is algebraically integrable and the first integrals of $\mathscr{F}$ are locally of finite type over $k$. 
Let $Y = \mathrm{Spec}\,\Gamma \left( X, \mathscr{O}_X^{\mathscr{F}}\right)$ and suppose that the induced morphism $\pi : X \rightarrow Y$ satisfies the properties of $\mathscr{F}$-stability.
Then, for all $x \in X$,
the induced morphism of local rings
$$ \mathscr{O}_{Y, \pi(x)} \rightarrow \mathscr{O}_{X,x}^{\mathscr{F}} $$
is an isomorphism. 
\end{lemma}

\begin{proof}
Both $\mathscr{O}_{X,x}^{\mathscr{F}}$ and $\mathscr{O}_{Y, \pi(x)}$ are subrings of the domain $\mathscr{O}_{X,x}$.
This implies that $\mathscr{O}_{Y, \pi(x)} \rightarrow \mathscr{O}_{X,x}$ is an injection.
It suffices to show it is a surjection. 
The proof is divided into three steps. Firstly, it is shown that the morphism $\mathscr{O}_{Y, \pi(x)} \rightarrow \mathscr{O}_{X,x}$ is generically finite. Secondly that it is birational and lastly, flatness is used to show that it is an isomorphism.
In the following, the results in \S \ref{subsec:first_integrals_two} can be used since, by hypothesis, $\mathscr{F}$ is locally free, hence torsion-free. \par

\textbf{Generically finite.}
Let $U \subseteq X$ be an affine open subset of $X$ containing $x$. 
For ease of notation, write $A^{\prime} = \Gamma \left( X, \mathscr{O}_X^{\mathscr{F}}\right)$, $A = \Gamma \left( U, \mathscr{O}_X^{\mathscr{F}}\right)$, $M = \mathscr{O}_{X,\eta}^{\mathscr{F}}$ and $B = \Gamma \left( X, \mathscr{O}_X\right)$. 
Because $B$ is a domain, $A$ and $A^{\prime}$, being subrings of $B$, are domains. 
Observe that there are field inclusions
$$ R(A^{\prime}) \rightarrow R(A) \rightarrow M \rightarrow R(B).$$
Furthermore
\begin{align*}
\mathrm{rank} \, \mathscr{F} &\leq \mathrm{tr\,deg}_{M} R(B) & \text{(Lemma \ref{lem:dim_first_integrals})} \\
&\leq \mathrm{tr\,deg}_{R(A)} R(B) & \\
&\leq\mathrm{tr\,deg}_{R(A^{\prime})} R(B) & \\
&= \mathrm{rank} \, \mathscr{F}. & \text{(Property \ref{stable:dimension})}
\end{align*}
Hence
\begin{align}
\label{eq:generically_finite}
\mathrm{tr\,deg}_{R(A^{\prime})} R(A) = 0
\end{align}

\textbf{Birationally equivalent.}
Let $U \subseteq X$ be an affine open subset of $X$ containing $x$. 
For ease of notation, write $A^{\prime} = \Gamma \left( X, \mathscr{O}_X^{\mathscr{F}}\right)$ and $A = \Gamma \left( U, \mathscr{O}_X^{\mathscr{F}}\right)$. 
Then the restriction morphism $A^{\prime} \rightarrow A$
induces an isomorphism of fields
\begin{align}
\label{eq:birational_equiv}
R(A^{\prime}) = R(A)
\end{align}
Indeed, applying (\ref{eq:generically_finite}) yields that $R(A^{\prime}) \rightarrow R(A)$ is an algebraic field extension and applying Proposition \ref{prop:birational_equivalence} gives the result. But taking the limit over all affine open subsets containing $x$ gives
\begin{align*}
R \left(\mathscr{O}_{X,x}^{\mathscr{F}} \right)
&= R \left( \lim_{\substack{\longrightarrow \\ U \ni x}} \mathscr{O}_X^{\mathscr{F}}(U) \right)
& \text{Lemma \ref{lem:localisation_well_defined}} \\
&= \lim_{\substack{\longrightarrow \\ U \ni x}}
R \left( \Gamma \left(U, \mathscr{O}_X^{\mathscr{F}}\right) \right)
& \substack{\text{commutativity} \\ \text{of localisations and colimits}} \\
&= R \left( \Gamma \left( X, \mathscr{O}_X^{\mathscr{F}}\right) \right) 
& (\ref{eq:birational_equiv}) \\
&= R \left( \mathscr{O}_{Y,\pi(x)} \right).\\
\end{align*}

\textbf{Isomorphic.} 
For ease of notation, write $A^{\prime} = \mathscr{O}_{Y, \pi(x)}$, $A = \mathscr{O}_{X,x}^{\mathscr{F}}$ and $B = \mathscr{O}_{X,x}$.
Let $a \in A$.
It is shown that it is in the image of $A^{\prime} \rightarrow A$.
By the second part of the proof, there exists $g$ and $h$ in $A^{\prime}$ such that $h a = g$.
Because $A^{\prime}$ is a domain, showing that $a \in A^{\prime}$ amounts to showing that $g \in (h) \subseteq A^{\prime}$.
To this end, consider the principal ideal $(h) \subseteq A^{\prime}$.
Because $A^{\prime} \rightarrow B$ is a flat morphism of local rings, it is faithfully flat.
This implies that
$$ (h) B \cap A^{\prime} = (h) \subseteq A^{\prime}.$$
But then $g \in A^{\prime}$ and $g \in (h) B$ since $g = h a$. This shows $g \in (h)$.
\end{proof}

\subsection{Generic Semistability}
\label{subsec:stability_generic}

This subsection shows that the generic point is stable.
A key step is to show that the generic fibre of the inclusion of first integrals into the regular functions is geometrically irreducible.
This is a consequence of Proposition \ref{prop:alg_closed}.

\begin{proposition}[The generic point is stable]
\label{prop:generic_stability}
Let $k$ be a field of characteristic zero and let $(X, \mathscr{F})$ be a normal integral foliated space locally of finite type over $k$.
Suppose that $\mathscr{F}$ is algebraically integrable, torsion-free and its first integrals are locally of finite type over $k$. 
Let $\eta$ be generic point of $X$.
Then $\eta$ is $\mathscr{F}$-stable.
\end{proposition}

\begin{proof}
Let $\iota : U \rightarrow X$ be an affine open subset such that the induced $k$-morphism
$$ U = \mathrm{Spec}\, \Gamma \left( U, \mathscr{O}_X \right) \rightarrow \mathrm{Spec}\, \Gamma \left( U, \mathscr{O}_X^{\mathscr{F}} \right) = V $$
gives an isomorphism of locally free sheaves
$$ \Omega_{U/V}^1 \xrightarrow{\sim} \iota^* \mathscr{F}.$$
To get such an open set, apply Lemma \ref{lem:traditional_algebraic_integrability}.
For ease of notation, replace $X$ by $U$.
Note that any affine open subset contained in $X$ also satifies the property above. \par

Now, the morphism
$$ \Gamma \left( X, \mathscr{O}_X^{\mathscr{F}} \right) \rightarrow \Gamma \left( X, \mathscr{O}_X \right) $$
is of finite type.
This follows from the fact that $k \rightarrow \Gamma \left( X, \mathscr{O}_X \right)$ is of finite type and \cite[\href{https://stacks.math.columbia.edu/tag/00F4}{Lemma 00F4}]{stacks-project}. 
Hence, by generic flatness (\cite[\href{https://stacks.math.columbia.edu/tag/051T}{Lemma 051T}]{stacks-project}), there exists an $f \in \Gamma \left( X, \mathscr{O}_X^{\mathscr{F}} \right)$ such that 
$$ \Gamma \left( X, \mathscr{O}_X^{\mathscr{F}} \right)_f \rightarrow \Gamma \left( X, \mathscr{O}_X \right)_f $$
is flat and of finite presentation.
Now, by Lemma \ref{lem:distinguished_first_integrals}, $\Gamma \left( X, \mathscr{O}_X^{\mathscr{F}} \right)_f = \Gamma \left( X_f, \mathscr{O}_X^{\mathscr{F}} \right)$.
Therefore, up to replacing $X$ by its affine open subset $X_f$, it may be assumed that
\begin{align}
\label{eq:first_integrals}
\Gamma \left( X, \mathscr{O}_X^{\mathscr{F}} \right) \rightarrow \Gamma \left( X, \mathscr{O}_X \right)
\end{align}
is flat of finite presentation and its relative sheaf of differentials is locally free of rank equal to $\mathrm{rank}\, \mathscr{F}$.
Now, the sheaf of differentials of the fibres of (\ref{eq:first_integrals}) is locally free.
Given that the characteristic of the base field is zero, this implies that the fibres are smooth (\cite[\href{https://stacks.math.columbia.edu/tag/04QN}{Lemma 04QN}]{stacks-project}).
Therefore (\ref{eq:first_integrals}) is smooth (\cite[\href{https://stacks.math.columbia.edu/tag/01V8}{Lemma 01V8}]{stacks-project}).
Thus \ref{stable:smooth} is shown. \par

It is clear, using \cite[\href{https://stacks.math.columbia.edu/tag/02G2}{Definition 02G2}]{stacks-project}, that (\ref{eq:first_integrals}) is of relative dimension $\mathrm{rank}\,\mathscr{F}$.
Thus \ref{stable:dimension} is satisfied. \par

For ease of notation, write $A = \Gamma \left( X, \mathscr{O}_X^{\mathscr{F}} \right)$, $B = \Gamma \left( X, \mathscr{O}_X \right)$ and $F = \Gamma \left( X, \mathscr{F}\right)$
By construction, the field of fraction of $A$, $R(A)$, is equal to $\mathscr{O}_{X, \eta}^{\mathscr{F}}$.
Clearly the field of fractions of $B$, $R(B) = \mathscr{O}_{X, \eta}$.
It is now shown that that the generic fibre $B \otimes_A R(A)$ is geometrically irreducible over $R(A)$.
To see this, note that the field of fractions of $B \otimes_A R(A)$ is $R(B)$.
Now, by Proposition \ref{prop:alg_closed} applied to the directed space $\left( \mathrm{Spec}\, R(B),  F \otimes_B R(B) \right)$, the field extension
$$ R(A) = \mathscr{O}_{X, \eta}^{\mathscr{F}} \rightarrow \mathscr{O}_{X, \eta} = R(B)$$
is algebraically closed.
Therefore $R(A) \rightarrow R(B)$ is geometrically irreducible (\cite[\href{https://stacks.math.columbia.edu/tag/037P}{Lemma 037P}]{stacks-project}).
Since $R(B)$ is the localisation of the generic fibre $B \otimes_A R(A)$ at its generic point, $B \otimes_A R(A)$ is geometrically irreducible over $R(A)$ (\cite[\href{https://stacks.math.columbia.edu/tag/054Q}{Lemma 054Q}]{stacks-project}).
By Property \ref{stable:smooth}, $A \rightarrow B$ is a morphism of finite type and $A$ is an integral domain.
Therefore \cite[\href{https://stacks.math.columbia.edu/tag/0559}{Lemma 0559}]{stacks-project} implies that there exists an $f \in A$ such that $A_f \rightarrow B_f$ has geometrically irreducible fibres, hence Property \ref{stable:connected} is satisfied.
Observe that $B_f = B \otimes_A A_f$ and since Property \ref{stable:smooth} is preserved by base change, $A_f \rightarrow B_f$ satisfies Property \ref{stable:smooth}.
Property \ref{stable:dimension} is also preserved.
Let $U = X_f \subseteq X$ be the affine open set where $f$ does not vanish. Then, again by Lemma \ref{lem:distinguished_first_integrals}
$$ A_f = \Gamma \left( X, \mathscr{O}_X^{\mathscr{F}} \right)_f = \Gamma \left( U, \mathscr{O}_X^{\mathscr{F}} \right).$$
Clearly $B_f = \Gamma \left( U, \mathscr{O}_X \right)$ and the Proposition is proven.
\end{proof}

\section{Main Theorems}
\label{sec:main_theorems}

\begin{theorem}[Main Theorem I]
\label{thm:main1}
Let $k$ be a field of characteristic zero and let $(X, \mathscr{F})$ be a normal integral foliated space locally of finite type over $k$.
Suppose that $\mathscr{F}$ is algebraically integrable and its first integrals are locally of finite type over $k$. 
Let $X^{\mathrm{s}}$ denote the set of $\mathscr{F}$-stable points. 
Then $X^{\mathrm{s}}$ is a non-empty open set and there exists a geometric (and categorical) quotient $\pi : X^{\mathrm{s}} \rightarrow Y$ of $X^{\mathrm{s}}$ by $\mathscr{F}$, where $Y$ is a normal integral scheme locally of finite type over $k$.
\end{theorem}

\begin{proof}
By Remark \ref{rem:stable_open}, $X^{\mathrm{s}}$ is open.
By Proposition \ref{prop:generic_stability}, $X^{\mathrm{s}}$ is non-empty.
By Proposition \ref{prop:stable_geometric quotient}, for every $x \in X^{\mathrm{s}}$, there exists an open set $U \subseteq X$ containing $x$ and a geometric quotient $\pi_U : U \rightarrow V$.
By Proposition \ref{prop:glue_geometric_quotient}, these morphisms can be glued to obtain a geometric quotient $\pi : X^{\mathrm{s}} \rightarrow Y$.
By Proposition \ref{prop:geometric_categorical_quotient}, $\pi$ is a categorical quotient.
Finally, again using Proposition \ref{prop:stable_geometric quotient}, it follows that $Y$ can be covered by affine open sets of the form $\mathrm{Spec} \, \Gamma \left(U, \mathscr{O}_X^{\mathscr{F}} \right)$ for some affine open sets $U \subseteq X$.
This implies that $Y$ is locally of finite type, reduced and normal (Lemma \ref{lem:normal_domains}.
$Y$ is also irreducible, since it is the image of a surjective morphism whose source is irreducible (\cite[\href{https://stacks.math.columbia.edu/tag/0379}{Lemma 0379}]{stacks-project}), hence $Y$ is integral.
\end{proof}

\begin{theorem}[Main Theorem II]
\label{thm:main2}
Let $k$ be a field of characteristic zero and let $(X, \mathscr{F})$ be a normal integral foliated space locally of finite type over $k$.
Suppose that $\mathscr{F}$ is algebraically integrable and $\mathrm{corank} \,\mathscr{F} \leq 2$. 
Let $X^{\mathrm{s}}$ denote the set of $\mathscr{F}$-stable points. 
Then $X^{\mathrm{s}}$ is a non-empty open set and there exists a geometric (and categorical) quotient $\pi : X^{\mathrm{s}} \rightarrow Y$ of $X^{\mathrm{s}}$ by $\mathscr{F}$, where $Y$ is a normal integral scheme locally of finite type over $k$.
\end{theorem}

\begin{proof}
Up to replacing $X$ by the open set where $\mathscr{F}$ is locally free, it may be assumed that $\mathscr{F}$ is locally free on $X$.
Indeed, Lemma \ref{lem:stable_smooth} shows that $X^{\mathrm{s}}$ is contained in the open set where $\mathscr{F}$ is locally free.
In particular, $\mathscr{F}$ is torsion-free and Corollary \ref{cor:fin_gen} guarantees that its first integrals are locally of finite type over $k$.
Theorem \ref{thm:main1} can then be applied to find a geometric (and categorical) quotient $\pi : X^{\mathrm{s}} \rightarrow Y$.
\end{proof}

\bibliography{Quotient_Spaces2}

\providecommand{\bysame}{\leavevmode\hbox to3em{\hrulefill}\thinspace}
\providecommand{\MR}{\relax\ifhmode\unskip\space\fi MR }
\providecommand{\MRhref}[2]{%
  \href{http://www.ams.org/mathscinet-getitem?mr=#1}{#2}
}
\providecommand{\href}[2]{#2}
\begin{thebibliography}{10}

\bibitem{MR1863738}
Jean-Beno\^{\i}t Bost, \emph{Algebraic leaves of algebraic foliations over
  number fields}, Publ. Math. Inst. Hautes \'{E}tudes Sci. (2001), no.~93,
  161--221. \MR{1863738}

\bibitem{MR0274237}
N.~Bourbaki, \emph{\'{E}l\'{e}ments de math\'{e}matique. {A}lg\`ebre.
  {C}hapitres 1 \`a 3}, Hermann, Paris, 1970. \MR{0274237}

\bibitem{MR3445519}
Jean-Pierre Demailly, \emph{Towards the {G}reen-{G}riffiths-{L}ang conjecture},
  Analysis and geometry, Springer Proc. Math. Stat., vol. 127, Springer, Cham,
  2015, pp.~141--159. \MR{3445519}

\bibitem{MR1917641}
Gene Freudenburg, \emph{A survey of counterexamples to {H}ilbert's fourteenth
  problem}, Serdica Math. J. \textbf{27} (2001), no.~3, 171--192. \MR{1917641}

\bibitem{MR1017286}
Xavier G\'{o}mez-Mont, \emph{Integrals for holomorphic foliations with
  singularities having all leaves compact}, Ann. Inst. Fourier (Grenoble)
  \textbf{39} (1989), no.~2, 451--458. \MR{1017286}

\bibitem{MR1432041}
Se\'{a}n Keel and Shigefumi Mori, \emph{Quotients by groupoids}, Ann. of Math.
  (2) \textbf{145} (1997), no.~1, 193--213. \MR{1432041}

\bibitem{MR3842065}
Frank Loray, Jorge~Vit\'{o}rio Pereira, and Fr\'{e}d\'{e}ric Touzet,
  \emph{Singular foliations with trivial canonical class}, Invent. Math.
  \textbf{213} (2018), no.~3, 1327--1380. \MR{3842065}

\bibitem{MR3128985}
Michael McQuillan and Daniel Panazzolo, \emph{Almost \'{e}tale resolution of
  foliations}, J. Differential Geom. \textbf{95} (2013), no.~2, 279--319.
  \MR{3128985}

\bibitem{MR1468476}
Yoichi Miyaoka and Thomas Peternell, \emph{Geometry of higher-dimensional
  algebraic varieties}, DMV Seminar, vol.~26, Birkh\"{a}user Verlag, Basel,
  1997. \MR{1468476}

\bibitem{MR1304906}
D.~Mumford, J.~Fogarty, and F.~Kirwan, \emph{Geometric invariant theory}, third
  ed., Ergebnisse der Mathematik und ihrer Grenzgebiete (2) [Results in
  Mathematics and Related Areas (2)], vol.~34, Springer-Verlag, Berlin, 1994.
  \MR{1304906}

\bibitem{MR82725}
Masayoshi Nagata, \emph{A general theory of algebraic geometry over {D}edekind
  domains. {I}. {T}he notion of models}, Amer. J. Math. \textbf{78} (1956),
  78--116. \MR{82725}

\bibitem{MR88034}
\bysame, \emph{A treatise on the 14-th problem of {H}ilbert}, Mem. Coll. Sci.
  Univ. Kyoto Ser. A. Math. \textbf{30} (1956), 57--70. \MR{88034}

\bibitem{stacks-project}
The {Stacks project authors}, \emph{The stacks project},
  \url{https://stacks.math.columbia.edu}, 2021.

\end{thebibliography}
\bibliographystyle{amsplain}

\end{document}